\documentclass[12pt]{amsart}

\usepackage[left=3cm,top=2.5cm,bottom=2.5cm,right=3cm]{geometry}
\usepackage{times}
\usepackage{amssymb}
\usepackage{graphicx,xspace}
\usepackage{epsfig}
\usepackage{enumitem}
\usepackage[usenames,dvipsnames]{xcolor}
\usepackage{tikz}
\usepackage{tikz-cd}
\usepackage{gensymb}
\usepackage[T1]{fontenc}
\usepackage[utf8]{inputenc}
\usepackage{bm}
\usepackage[boxsize=1em]{ytableau}
\usepackage[presets={abc,vec-cev}]{letterswitharrows}
\usepackage[config, labelfont={normalsize}]{caption,subfig}
\usepackage{mathrsfs}
\definecolor{green}{RGB}{0,127,0}
\definecolor{red}{RGB}{191,0,0}
\usepackage[colorlinks,cite color=red,link color=green,pagebackref=true]{hyperref}
\usepackage{todonotes}

\tikzset{
	arrow/.pic={\draw[every arrow/.try,->,>=#1] (0,0) -- +(.1pt,0);},
	pics/arrow/.default={triangle 90}
}

\usetikzlibrary{matrix,arrows,calc}
\usetikzlibrary{decorations.markings}
\usetikzlibrary{patterns}
\usetikzlibrary{snakes}
\usetikzlibrary{shapes}

\usepackage[capitalize]{cleveref}

\theoremstyle{plain}

\newtheorem{lemma}{Lemma}[section]
\newtheorem{theorem}[lemma]{Theorem}
\newtheorem{corollary}[lemma]{Corollary}
\newtheorem{proposition}[lemma]{Proposition}
\newtheorem{problem}[lemma]{Problem}

\newtheorem{conjecture}[lemma]{Conjecture}

\theoremstyle{remark}
\newtheorem{remark}[lemma]{Remark}
\newtheorem*{example}{Example}
\newtheorem{definition}[lemma]{Definition}

\newcommand{\svdots}{%
	\vbox{
		\scriptsize \baselineskip 2pt \lineskiplimit 0pt
		\hbox {.}\hbox {.}\hbox {.}\kern-0.75pt
	}%
}

\newcommand{\shdots}{%
	\hbox{
		\scriptsize \baselineskip 20pt \lineskiplimit 2pt
		\hbox {.}\hbox {.}\hbox {.}\kern2.95pt
	}%
}

\DeclareMathOperator{\LLT}{LLT}

\newcommand{\A}{\mathcal{A}}

\newcommand{\N}{\mathbb{N}}

\newcommand{\I}{\mathcal{I}}

\newcommand{\Z}{\mathbb{Z}}

\newcommand{\QQ}{\mathbb{Q}}

\newcommand{\xx}{\bm{x}}

\newcommand{\kumu}{\kappa}

\newcommand{\PPP}{\mathcal{P}}

\DeclareMathOperator{\RSST}{RSST}

\DeclareMathOperator{\PF}{PF}
\def\la{\lambda}

\def\ka{\kappa}

\DeclareMathOperator{\Tu}{Tutte}

\DeclareMathOperator{\cc}{cc}

\newcommand{\lla}{\bm{\la}}

\newcommand{\nnu}{\bm{\nu}}

\DeclareMathOperator{\SYT}{SYT}
\DeclareMathOperator{\SSYT}{SSYT}

\DeclareMathOperator{\id}{id}

\DeclareMathOperator{\inv}{inv}
\DeclareMathOperator{\Inv}{Inv}

\DeclareMathOperator{\maj}{maj}
\DeclareMathOperator{\sh}{sh}

\DeclareMathOperator{\Des}{Des}
\DeclareMathOperator{\cospin}{cospin}

\DeclareMathOperator{\Mac}{Mac}

\DeclareMathOperator{\QSym}{QSym}

\def\uu{\bm{u}}

\author[M.~Dołęga]{Maciej Dołęga}
\address{
Institute of Mathematics, 
Polish Academy of Sciences, 
ul. Śniadeckich 8, 
00-956 Warszawa, Poland}
\email{mdolega@impan.pl}

\author[M.~Kowalski]{Maciej Kowalski}
\address{
Institute of Mathematics, 
Polish Academy of Sciences, 
ul. Śniadeckich 8, 
00-956 Warszawa, Poland}
\email{mkowalski@impan.pl}

 \thanks{
The authors are supported by {\it Narodowe Centrum Nauki}, grant UMO-2017/26/D/ST1/00186.}

\title [LLT cumulants and graph coloring]{LLT cumulants and graph coloring}

\begin{document}

\maketitle

\begin{abstract}
	The purpose of this note is to introduce a new family of
	quasi-symmetric functions called \emph{LLT
	cumulants} and discuss its properties. We define LLT cumulants using the algebraic
	framework for conditional cumulants and we prove that the Macdonald
	cumulant has an explicit positive expansion in terms of LLT cumulants of ribbon shapes,
	generalizing the classical decomposition of Macdonald polynomials. We also find a natural combinatorial interpretation of
	the LLT cumulant of a given directed graph as a weighted generating
	function of colorings of its subgraphs.
	
	We use this graph theoretical framework to prove
	various positivity results. This includes monomial positivity,
	positivity in fundamental quasisymmetric functions and related
	positivity of the coefficients of Schur polynomials indexed by hook
	shapes. We also prove $e$-positivity for vertical-shape LLT cumulants, after the shift of
	variable $q \to q+1$, which refines a recent result of Alexandersson
	and Sulzgruber. All these results give evidence towards
	Schur-positivity of LLT cumulants, which we conjecture here. We prove that this conjecture implies
	Schur-positivity of Macdonald cumulants, and we give more evidence by
	proving the conjecture for LLT cumulants of melting lollipops that
        refines a recent result of Huh, Nam and Yoo.
\end{abstract}

\ytableausetup{smalltableaux}

\section{Introduction}

In 1988, Macdonald~\cite{Macdonald1988} introduced his celebrated
two-parameter symmetric functions and conjectured that when expanded in the
basis of Schur symmetric functions, their coefficients have a remarkable
property: they seem to be polynomials in two deformation parameters $q,t$
with nonnegative integer coefficients. Since then, a broad community
working on the symmetric functions theory devoted themselves to prove
Macdonald's conjecture, which resulted in a huge development of the
field.

In 1995, Lapointe and Vinet~\cite{LapointeVinet1995} proved that the coefficients of Jack
symmetric functions expanded in the monomial basis are polynomials in
the deformation parameter $\alpha$ with integer coefficient. Two
years later, Knop and Sahi~\cite{KnopSahi1997} found an explicit positive formula for
this expansion. Since Jack symmetric functions are
a limit case of Macdonald
symmetric functions, these results inspired further research
and shortly afterwards, the polynomiality
of the coefficients of Macdonald
polynomials was proved
independently and almost simultaneously (using different
approaches) in five different
papers~\cite{Sahi1996,GarsiaTesler1996,LapointeVinet1997,Knop1997,KirillovNoumi1998}. An
affirmative answer to Macdonald's original conjecture was finally
released in a beautiful and difficult paper of
Haiman~\cite{Haiman2001}, who was able to relate Macdonald's question
to a question about the geometry of Hilbert schemes of points in the
complex plane, to which he gave an affirmative answer. Even though this result built new bridges between various fields of mathematics, it did not provide an explicit combinatorial formula explaining
Schur-positivity. Regardless, it generated new research directions related with the structure of Macdonald and related symmetric functions.

In 2005, Haglund, Haiman and Loehr~\cite{HaglundHaimanLoehr2005}
found an explicit combinatorial formula for Macdonald
polynomials, lifting Knop and Sahi's
formula to the two-parameter world of Macdonald, and relating
Macdonald polynomials with another family of symmetric functions
introduced by Leclerc, Lascoux and Thibon in
1997~\cite{LascouxLeclercThibon1997}, and later conveniently named
\emph{LLT polynomials}. Haglund, Haiman and Loehr noticed~\cite{HaglundHaimanLoehr2005} that
Macdonald polynomials can be naturally decomposed as a positive
combination of LLT polynomials, so proving Schur-positivity for
LLT polynomials would give yet another proof of the famous
conjecture of Macdonald. This was done by Grojnowski and
Haiman~\cite{GrojnowskiHaiman2007}, who related LLT polynomials with
the Kazhdan-Lusztig theory in a much more general setting than what was
done before~\cite{LeclercThibon2000}, and therefore proved the Schur-positivity of LLT polynomials
indexed by arbitrary skew-shapes (see \cref{sec:LLTMacpolynomials} for the
details and all the necessary definitions).

\vspace{10pt}

In 2017, the first author together with F\'eray~\cite{DolegaFeray2017}
introduced \emph{Jack cumulants} as a tool to approach a fascinating
open problem in the theory of symmetric functions known as \emph{the
$b$-conjecture} (posed by Goulden and
Jackson~\cite{Gouldenjackson1996}), which relates Jack symmetric
functions with a weighted generting function of graphs embedded into
surfaces (and which, despite some recent progress~\cite{ChapuyDolega2022}, is still wide open). The notion of Jack cumulants naturally extends to
\emph{Macdonald cumulants} the same way as Jack polynomials can be seen
as the limit case of Macdonald polynomials. The first author with F\'eray observed
conjecturally that the coefficients of Macdonald cumulants seem to be
polynomials, which was later proved in~\cite{Dolega2017} and further
improved in~\cite{Dolega2019}, where an explicit positive
combinatorial formula for the Macdonald cumulants was proved. This rich combinatorial structure of
Macdonald cumulants naturally calls for investigating the expansion in
the Schur basis: extensive computer simulations performed by the first author~\cite{Dolega2017}
have led him to believe that a more general version of the original
question of Macdonald is true: the coefficients of the Schur expansion
of Macdonald cumulants are polynomials in $q,t$ with nonnegative
integer coefficients. We were recently informed that the
logarithm of a partition function for Macdonald polynomials was
considered by Hausel, Letellier and Rodriguez Villegas
~\cite{HauselLetellierRodriguezVillegas2011}, who conjectured its
monomial positivity interpreted as the mixed Hodge polynomials of character varieties. The notion of Macdonald cumulants
appears naturally in the decomposition of the logarithm of the
partition function, and the recent
work~\cite{AbdelgadirMellitRodriguezVillegas2019} (following~\cite{HauselSturmfels2002,Crawley-BoeveyVandenBergh2004}) exhibits that the
Poincar\'e polynomial of Nakajima quiver variaties (which can be seen
as a special case of the aforementioned conjecture) is given by the specialization of the Tutte polynomial $\Tu_G(1,q)$,
which is the same phenomenon as in our combinatorial interpretation of
Macdonald cumulants~\cite{Dolega2019}. All this gives yet additional
motivation for studying the combinatorial structure of Macdonald cumulants.

\vspace{10pt}

The main purpose of this note is to take a step in this direction
by introducing the notion of LLT cumulants. There are two natural
motivations for introducing them:
\begin{itemize}
\item
  by analogy with the decomposition of Macdonald
  polynomials into LLT polynomials, we show that the same phenomenon occurs at the
  level of higher cumulants: \emph{Macdonald cumulants} can be naturally
  expressed as a positive linear combination of \emph{LLT cumulants}
  -- see \cref{theo:MacCumuIntoLLTCumu};
\item
  in contrast to the purely algebraic definition of Macdonald
  cumulants inspired by the theory of conditional cumulants, we show that LLT cumulants (a priori defined using the same abstract framework)
  can be equivalently defined purely combinatorially as
  graph colorings -- see \cref{thm:cumasconnected}. In particular, it is
  natural to study a general class of graph
  colorings which contains LLT
  polynomials and LLT cumulants, and that allows to treat certain LLT-specific phenomena
  in a more general graph-theoretical sense -- see~\cref{sec:GraphCol}.
\end{itemize}

There are several applications of the aforementioned results. We start
by developing the theory of $q$-partial cumulants, which generalize
the $G$-inversion polynomials or, equivalently, the generating series of
$G$-parking
functions, which is also equal to the evaluation of the Tutte
polynomial $\Tu_G(1,q)$ (see~\cref{subsec:cumu}), and we prove a positivity result for these
cumulants (see~\cref{theo:G-cumu}). This result is crucial for
proving~\cref{theo:LLT->Mac} which says that Schur positivity of the
cospin LLT cumulant (that we state as~\cref{conj:LLTSchurPos}) implies
Schur positivity of Macdonald cumulants conjectured
in~\cite{Dolega2017}. In \cref{sec:GraphCol}, we introduce certain
digraphs that we call \emph{LLT graphs}, and we show that every LLT
polynomial is a weighted generating function of LLT graph
colorings. We describe the ring generated by these LLT graphs and we
prove that the LLT cumulant of an $r$-colored LLT graph
$(G,f)$ has a
natural interpretation as a weighted generating function of colorings
of all \emph{$f$-connected} subgraphs of $G$
(see~\cref{def:kaconnected} and the preceding paragraph for
the precise definition of $f$-connectedness and LLT cumulants of $r$-colored LLT graphs). We obtain this interpretation by studying
certain relations between colorings of various LLT graphs.

It is worth mentioning that recently, various authors have already proven many interesting results concerning positivity of LLT
polynomials and they heavily relied on some relations between
them~\cite{Lee2021,HuhNamYoo2020,AbreuNigro2021,AlexanderssonSulzgruber2022,Tom2021}. Our interpretation of LLT polynomials and LLT cumulants proves that the graph-theoretical point of
view is very natural and a characterization of all possible relations might
potentially be achieved pushing these studies further in the future
(see~\cref{rem:LLTRel}). We use our framework to refine some of the previous
positivity results, which gives evidence towards~\cref{conj:LLTSchurPos}:
\begin{itemize}
\item we prove that the coefficients of LLT cumulants of an $r$-colored LLT graph $(G,f)$
  in the quasi-symmetric monomial basis are polynomials in $q$ with
  nonnegative integer coefficients and we provide their explicit combinatorial
  interpretation (see~\cref{thm:kumuMonoPos}). This result is a
  refinement of the combinatorial formula for Macdonald polynomials~\cite{Dolega2019};
\item
  we deduce an analogous result for the fundamental quasisymmetric
  basis and using standard procedures, we deduce positivity of the
  coefficients of Schur basis indexed by hooks (see~\cref{theo:LLTFundamental,theo:LLTShurHooks});
\item
  we prove that LLT polynomials
  considered after the shift $q \to q+1$ naturally decompose as a sum
  of products of LLT cumulants. In the special case of vertical-strips,
  we deduce from the recent result of Alexandersson and
  Sulzgruber~\cite{AlexanderssonSulzgruber2022} a positive
  combinatorial formula for LLT cumulants in the basis of elementary
  functions (see~\cref{theo:e-positivity});
  \item
  we prove Schur positivity of LLT cumulants of $r$-colored lollipop
  graphs, generalizing previous result of Huh, Nam and Yoo~\cite{HuhNamYoo2020} (see~\cref{subsec:UniLLT} for the
  definitions and~\cref{theo:Lollipop} for the result).
\end{itemize}

Our paper is organized as follows: in~\cref{sec:LLTCum}, we review the
necessary background on Macdonald and LLT polynomials and on
cumulants. Then we introduce $q$-partial cumulants, we state our
main~\cref{conj:LLTSchurPos}, and we prove that it implies Schur
positivity of Macdonald cumulants. \cref{sec:GraphCol} is devoted to the
study of LLT graphs and weighted generating functions
of their colorings that we introduce. In~\cref{subsec:LLTGraphs}, we give a combinatorial
interpretation of LLT cumulants in the graph-theoretical framework and
in~\cref{subsec:Positivity}, we prove various positivity results
supporting~\cref{conj:LLTSchurPos}. In~\cref{sec:Coclude},
we conclude with comments and questions related with~\cref{conj:LLTSchurPos} and, in
particular, further partial results including Schur positivity for LLT
cumulants of $r$-colored melting lollipops.

\section{Macdonald cumulants and expansion in LLT cumulants}
\label{sec:LLTCum}

We use French convention for drawing Young diagrams, i.e.~ the largest
row is at the bottom and the largest column is on the left hand side.

\subsection{LLT and Macdonald polynomials} \label{sec:LLTMacpolynomials}

Let $\nnu = (\la^1/\mu^1,\dots,\la^\ell/\mu^\ell)$ be an $\ell$-tuple of
skew Young diagrams (and denote $\ell(\nnu) := \ell$). For each box $\square = (x,y) \in \la^i/\mu^i$, we define
its content $c(\square) = x-y$ and its shifted content as
$\tilde{c}(\square) = \ell c(\square)+i-1$. We say that a box $\square \in \nnu$
\emph{attacks} a box $\square' \in \nnu$ if $0<
\tilde{c}(\square')-\tilde{c}(\square) < \ell$. Let $T$ be a filling
of cells of the diagrams in $\nnu$. If for each $i \in [1..\ell]$ the
entries in $\la^i/\mu^i$ are weakly increasing in rows (from left to
right) and strictly increasing in columns (from bottom to top), we say
that $T$ is a \emph{semistandard} filling, and we denote it by $T \in
\SSYT(\nnu)$. Finally, we call a pair of boxes $\square,\square'
\in \nnu$ an \emph{inversion} of $T$ if $T(\square) > T(\square')$ and $\square$ attacks
$\square'$. We denote the set of inversions of $T$ by $\Inv(T)$ and
its cardinality by $\inv(T)$.

\vspace{5pt}

\emph{LLT polynomial} $\LLT(\nnu)$ is the weighted generating series of $\SSYT(\nnu)$:

\begin{equation}
  \label{eq:LLT-def}
  \LLT(\nnu) = \sum_{T \in \SSYT(\nnu)}q^{\inv(T)}\xx^T,
  \end{equation}
  where $\xx^T := \prod_{\square \in \nnu} x_{T(\square)}$.

  \vspace{5pt}

  This definition was introduced in \cite{HaglundHaimanLoehr2005} and
  it is related to the original definition of Lascoux, Leclerc and
  Thibon \cite[Equation (26)]{LascouxLeclercThibon1997} by:

  \begin{equation}
  \label{eq:LLT-cospin-def}
  \LLT^{\cospin}(\nnu) = q^{-\min_{T \in \SSYT(\nnu)}\inv(T)}\sum_{T \in \SSYT(\nnu)}q^{\inv(T)}\xx^T,
\end{equation}
where $\LLT^{\cospin}(\nnu) =
        \tilde{G}^{(r)}_\rho(X;q)$ using notation from
        \cite[Equation (26)]{LascouxLeclercThibon1997}.

        \begin{remark} The shape $\rho$ is obtained from $\nnu$ via the Stanton--White algorithm \cite{StantonWhite1985}, and since we do not use the original version of $r$-ribbon tableaux in this article, we treat \cref{eq:LLT-def} as the definition and refer to
	\cite{StantonWhite1985,LascouxLeclercThibon1997} for those who are
	interested in the equivalent framework of $r$-ribbon tableaux.
\end{remark}

The statistic $\inv(T)-\min_{T \in \SSYT(\nnu)}\inv(T)$ can be
  realized as the cardinality of a subset $\Inv_{\cospin}(T)$ of $\Inv(T)$ due to
  \cite{SchillingShimozonoWhite2003} (in particular, $\min_{T \in
    \SSYT(\nnu)}\inv(T) = |\Inv(T)\setminus\Inv_{\cospin}(T)|$ for
  any $T \in
    \SSYT(\nnu)$). For a box $\square \in \nnu$, we denote
  by $\square_{\leftarrow}, \square_{\rightarrow}, \square_{\uparrow},
  \square_{\downarrow}$ the boxes which are lying directly to the
  left, right, up and down of the box $\square$, respectively. Define $\Inv_{\cospin}(T)$ as follows:
  \begin{align*}
    \Inv_{\cospin}(T) = \{&(\square,\square') \in
  \Inv(T): (\square'_{\uparrow},\square) \in
  \Inv(T) \text{ and the row coordinate} \\
    &\text{of $\square$ is weakly smaller
    than the row coordinate of $\square'$} \}.
    \end{align*}
  Here, the convention is
    that for $\square'_{\uparrow}\notin \nnu$ the pair
    $(\square'_{\uparrow},\square)$ is automatically an inversion.
    Then 
  \begin{equation}
  \label{eq:LLT-cospin-def2}
  \LLT^{\cospin}(\nnu) = \sum_{T \in \SSYT(\nnu)}q^{|\Inv_{\cospin}(T)|}\xx^T.
\end{equation} 

In the special case when $\nnu =
(\la^1/\mu^1,\dots,\la^\ell/\mu^\ell)$ is a sequence of
ribbon shapes, i.e.,~ connected skew shapes which do not contain a
shape of size $2\times 2$, we define a normalization
  \begin{equation}
  \label{eq:LLT-Mac-def}
  \LLT^{\Mac}(\nnu) = q^{-a(\nnu)}\sum_{T \in \SSYT(\nnu)}q^{\inv(T)}\xx^T,
\end{equation}
where
\[ a(\nnu) = \sum_{\square \in \Des(\nnu)}|\{\square':
c(\square')=c(\square), \tilde{c}(\square')>\tilde{c}(\square)\}|,\]
and an element of $\Des(\nnu)$ is a box in $\nnu$, which is lying
directly above another box in $\nnu$.

This particular choice of normalization is motivated by the
combinatorial formula of Haglund, Haiman and Loehr, for
Macdonald polynomials $\tilde{H}_{\lambda}^{(q,t)}$. It expresses
a Macdonald polynomial $\tilde{H}_{\lambda}^{(q,t)}$ as a sum of LLT
polynomials indexed by $\la_1$-tuples of shapes of sizes
$\la'_j$, $1 \leq j \leq \la_1$ where $\lambda'$ denotes the \emph{transpose} of $\lambda$, i.e., the diagram with $\lambda_1$ boxes in the first column, $\lambda_2$ boxes in the second column, etc. For our purposes, we treat the following
formula as the definition of Macdonald polynomials:

\begin{theorem}{\cite{HaglundHaimanLoehr2005}}
  \label{theo:HHL}
   For any partition $\lambda$ the following expansion holds true
    \begin{equation}
      \label{eq:MacIntoLLT}
      \tilde{H}_{\lambda}^{(q,t)} =
      \sum_{\nnu}t^{\maj(\nnu)}\LLT^{\Mac}_{\nnu},
    \end{equation}
    where we sum over all tuples of skew-partitions such that $\nnu_j$ is a ribbon of length $\la'_j$
    whose bottom, far-right cell has content
    $0$. 
  \end{theorem}

  The statistic $\maj$, which appears in \eqref{eq:MacIntoLLT}, is defined as follows:
  \begin{equation}
    \label{eq:maj}
    \maj(\nnu) := \sum_{i=1}^{\ell(\nnu)}\maj(\nnu_i) = \sum_{i=1}^{\ell(\nnu)}\sum_{\square \in
      \Des(\nnu_i)}|\{\square' \in \nnu_i: c(\square')<c(\square)|.
    \end{equation}

    \subsection{Cumulants}
    \label{subsec:cumu}

    The notion of cumulants was originally studied by Leonov \sloppy and
    Shiryaev~\cite{LeonovShiryaev1959} in the context of probability
    theory. Cumulants appear now in a wide variety of contexts, see~\cite[Chapter
    6]{JansonLuczakRucinski2000} for their role in studying random graphs and~\cite{NovakSniady2011} for a
    concise introduction to noncommutative probability
    theory and various types of cumulants. In what follows, we will be
    interested in the $q$-deformation of partial cumulants
    that appeared in~\cite{Dolega2017} and was inspired by the
    classical definition of conditional cumulants (see \cref{def:Cumu}).

    \begin{definition}
      \label{def:QCumu}
    Suppose
that $\A$ is an algebra over the fraction field $\QQ(q)$. Let $\uu := (u_I)_{I \subseteq V}$
    be a family of elements in $\A$,
    indexed by subsets of a finite set $V$.
    Then its {\em $q$-partial cumulants} are defined as follows.
        For any non-empty subset $I$ of $V$, set
        \begin{equation}
            \ka^{(q)}_I(\uu) = (q-1)^{1-|I|}\sum_{\substack{\pi \in \PPP(I) } } 
            (-1)^{|\pi|-1}(|\pi|-1)!  \prod_{B \in \pi} u_B.
        \label{EqDefCumulants}
    \end{equation}
  \end{definition}
  The sum runs over elements of the family $\PPP(I)$ of
  \emph{set-partitions} of $I$: a set-partition $\pi \in \PPP(I)$ is a
  set of disjoint subsets of $I$ whose union is equal to $I$ (so
  one can think that an element $\pi \in \PPP(I)$ is grouping elements
  of $I$ into disjoint subsets) and the number of elements of $\pi$ is
  denoted by $|\pi|$.

  \begin{definition}
    \label{def:Cumu}
Let $\A$ be a vector space with two different
commutative multiplicative
structures $\cdot$
and $\oplus$, which define two (different) algebra structures
on $\A$. For any $X_1,\dots,X_r \in \A$, we define the \emph{conditional
cumulant} $\ka(X_1,\dots,X_r) \in \A$ as the coefficient of $t_1\cdots t_r$ in the following formal power
series in $t_1,\dots,t_r$:
\begin{equation} 
\label{eq:cumu1}
\ka(X_1,\dots,X_r) := [t_1\cdots t_r] \log_{\cdot} \left(
  \exp_{\oplus}(t_1X_1+\cdots + t_rX_r)\right),
\end{equation}
where $\log_{\cdot}$ and $\exp_{\oplus}$ are defined in a standard way
with respect to multiplication given by $\cdot$ and $\oplus$,
respectively.
\end{definition}

With the above in mind, we get
\[ \log_{\cdot}(1+A) = \sum_{n \geq 1}\frac{(-1)^{n-1}A^{\cdot
      n}}{n}, \qquad \exp_{\oplus}(A) = \sum_{n \geq 0}\frac{A^{\oplus n}}{n!}.\]
Then, one can show that setting
\[ u_B := \bigoplus_{b \in B}X_b,\]
the $q$-partial cumulant $\ka^{(q)}_{[1..r]}(\uu)$ evaluated at $q=0$
coincides with the conditional cumulant $\ka(X_1,\dots,X_r)$ up to a
sign:
\[ \ka^{(0)}_{[1..r]}(\uu) = (-1)^{r-1}\ka(X_1,\dots,X_r).\]

Although the cumulants originate from the probability theory, the
$q$-deformation introduced here is also relevant in the context of certain graph
invariants, called \emph{inversion polynomials}. Let $G = (V,E)$ be a multigraph
(i.e. a graph with multiple loops and multiple edges allowed) and for any subset of vertices
$I
\subset V$ we denote by $e_I$ the number of edges in $G$ connecting
vertices in $I$. It was shown in \cite{Dolega2019} that for the family
$\uu$ defined by
\[ u_I := q^{e_I}\]
the asociated $q$-partial cumulant $\ka^{q}_{V}(\uu)$ is equal to the
$G$-inversion polynomial $\mathcal{I}_G(q)$ (which is also equal to the
evaluation of the Tutte polynomial
$\Tu_G(1,q)$ and to the generating series of $G$-parking function;
a fact that will not be used in this paper). In particular, it is a polynomial
in $q$ with nonnegative integer coefficients and it was used to prove
positivity results for the $q$-partial cumulants of Macdonald
polynomials; we postpone its precise definition to
\cref{subsubsec:Tutte}, where we use it to provide certain explicit
combinatorial formulae. In the following, we show another positivity property of cumulants constructed by using
multigraphs. This positivity property will be crucial for our
first applications.

Suppose that $\uu$ is a family as in \cref{def:QCumu}, and let $G$ be a multigraph with the vertex set $V$. Define the family $\uu^G$ by setting
  \[ u_I^G := q^{e_I}u_I\]
  for any subset $I \subset V$. Finally, for any set-partition $\pi
  \in \PPP(I)$, define a family $\uu(\pi) := (\tilde{u}_B)_{B \subset
    \pi}$ by setting $\tilde{u}_B := u_{\cup B}$ (note that for $B \subset\pi
\in \PPP(I)$, one has $\bigcup B \subset I$ so that $\uu(\pi)$ is well defined).

  \begin{theorem}
    \label{theo:G-cumu}
  The $q$-partial cumulant $\ka^{(q)}_I(\uu
  ^G)$ is a
  $q$-positive combination of the $q$-partial cumulants
  $\ka^{(q)}_{\pi}(\uu(\pi))$, where $\pi \in \PPP(I)$.
    \end{theorem}

    \begin{proof}
      We will prove the theorem by induction on $|I|$. For $|I|=1$, the statement
      is obvious so suppose that $|I| >1$. Strictly from the definition
      of the $q$-partial cumulant~\eqref{EqDefCumulants}, $\ka^{(q)}_I(\uu
  ^G)$ can be expressed as $\ka^{(q)}_I(\uu
  ^G)  = q^{\sum_{i \in I}e_{\{i\}}}\cdot\ka^{(q)}_I(\uu
  ^{G'})$, where $G'$ is the graph $G$ restricted to the vertices from
  $I$ and with all the loops removed. Indeed, for every set-partition
  $\pi \in \PPP(I)$, the summand $(-1)^{|\pi|-1}(|\pi|-1)!  \prod_{B \in \pi} q^{e_B}u_B$
  appearing in~\eqref{EqDefCumulants} can be rewritten as $q^{\sum_{i
      \in I}e_{\{i\}}}(-1)^{|\pi|-1}(|\pi|-1)!  \prod_{B \in \pi}
q^{e_B(G')}u_B,$
where $e_B(G')$ denotes the number of edges in $G'$ connecting
vertices in $B$ (which is the same as the number of edges in $G$ connecting
distinct vertices in $B$). We further decompose
$\kumu_I^{(q)}(\uu^G)$ as
\[ \kumu_I^{(q)}(\uu^G) = q^{\sum_{i \in I} e_i}\bigg
  (\kumu^{(q)}_I(\uu) + \Big(\ka^{(q)}_I(\uu
  ^{G'}) - \kumu^{(q)}_I(\uu) \Big)\bigg),\]
which is relevant for using the inductive hypothesis. Indeed, the
second term in this decomposition can be expressed as:
\begin{align*}
  \ka^{(q)}_I(\uu
  ^{G'}) - \kumu^{(q)}_I(\uu) &= (q-1)^{1-|I|}\sum_{\substack{\pi \in \PPP(I) } } 
            (-1)^{|\pi|-1}(|\pi|-1)!  \big(q^{\sum_{B \in
                \pi}e_B(G')}-1\big)\prod_{B \in \pi} u_B\\
      &=(q-1)^{2-|I|}\sum_{\substack{\pi
          \in\PPP(I),\\ |\pi| < |I|}} (-1)^{|\pi|-1}(|\pi|-1)! \left[\sum_{B \in
            \pi}e_B(G') \right]_q\prod_{B\in \pi}u_B,      
            \end{align*}
      where $[n]_q := \frac{q^n-1}{q-1} =
      \sum_{i=0}^{n-1}q^i$ is the standard numerical factor. Let $e(G')
      := e_{I}(G')$
      denote the number of edges in $G'$. In the
      following, we are
      going to construct set-partitions $\sigma_1,\dots, \sigma_{e(G')}
      \in \PPP(I)$ each consisting of precisely $|I|-1$ elements, and graphs $G_1,\dots, G_{e(G')}$ with $|I|-1$ vertices such that
      \begin{align}
        \label{eq:GraphEq}
        (q-1)^{2-|I|}\sum_{\substack{\pi
          \in\PPP(I),\\ |\pi| < |I|}} (-1)^{|\pi|-1}(|\pi|-1)! \left[\sum_{B \in
            \pi}e_B(G') \right]_q\prod_{B\in \pi}u_B =
       \sum_{k=1}^{e(G')} \kumu^{(q)}_{\sigma_k}(\uu(\sigma_k)^{G_k}),
      \end{align}
      which allows to conclude the proof using the inductive hypothesis.
      
      We arbitrarily order edges of $G'$ and for any $1 \leq i \leq
      e(G')$ we denote by $E_i(G')$ the
      set of the first $i$ edges in $G'$. Let $\{m,n\}$ be the set of
      endpoints of the $i$-th edge in $G'$. Define the graph $G_i$ as
      follows:
      \begin{itemize}
        \item its set of vertices is equal to the set partition
      $\sigma_i := \{\{m,n\},\{k\}\colon k\in I\setminus\{m,n\}\}$,
      which is the unique set partition of $I$ with
      $|I|-1$ elements, whose element of size two is equal to
      $\{m,n\}$. In other terms, there is precisely one vertex of
      $G_i$ equal to the set $\{m,n\}$, and every other
      vertex of
      $G_i$ is equal to the singleton $\{k\}$, where $k \in
      I\setminus\{m,n\}$. In particular, $G_i$ has precisely $|I|-1$
      vertices.
      \item For
      any elements $k,l \in I\setminus\{m,n\}$, the number of edges
      linking vertices $\{k\}$ and $\{l\}$ in $G_i$ is given by the number of
      edges in $E_i(G')$ with endpoints $\{k,l\}$.
      \item For each $k \in
     I\setminus\{m,n\}$, the number of edges linking vertices $\{k\}$
      and $\{m,n\}$ in $G_i$ is given by the number of
      edges in $E_i(G')$ with endpoints $\{k,m\}$ or
      $\{k,n\}$.
      \item Finally, the number of loops attached to vertex
      $\{m,n\}$ is given by the number of
      edges in $E_{i-1}(G')$ with endpoints $\{m,n\}$.
      \end{itemize}

      Let us prove by
      induction on $e(G')$ that the constructed graphs satisfy
      \cref{eq:GraphEq}. Clearly, when $G'$ has no edges, both
      hand sides of \cref{eq:GraphEq} are equal to $0$. Suppose that
      $e(G') > 0$ and let $G''$ denote the graph obtained from
      $G'$ by removing its largest edge $\{m,n\}$. Then
            \begin{align*}
        &(q-1)^{2-|I|}\sum_{\substack{\pi
          \in\PPP(I),\\ |\pi| < |I|}} (-1)^{|\pi|-1}(|\pi|-1)! \left[\sum_{B \in
            \pi}e_B(G') \right]_q\prod_{B\in \pi}u_B = \\ &(q-1)^{2-|I|}\sum_{\substack{\pi
          \in\PPP(I),\\ |\pi| < |I|}} (-1)^{|\pi|-1}(|\pi|-1)! \left[\sum_{B \in
            \pi}e_B(G'') \right]_q\prod_{B\in \pi}u_B + \\ &(q-1)^{2-|I|}\sum_{\pi
          \in\PPP(\sigma_{e(G')})} (-1)^{|\pi|-1}(|\pi|-1)! \prod_{B\in
                                                             \pi}q^{e_{\cup B} (G'')}
                                                             u_{\cup B}.
            \end{align*}
            By the inductive hypothesis, we have that
                        \begin{align*}
        (q-1)^{2-|I|}\sum_{\substack{\pi
          \in\PPP(I),\\ |\pi| < |I|}} (-1)^{|\pi|-1}(|\pi|-1)! \left[\sum_{B \in
            \pi}e_B(G'') \right]_q\prod_{B\in \pi}u_B = \sum_{k=1}^{e(G'')} \kumu^{(q)}_{\sigma_k}(\uu(\sigma_k)^{G_k}).
                        \end{align*}
                        Moreover, strictly from the construction of
                        $G_{e(G')}$, we have that $e_B(G_{e(G')}) =
                        e_{\cup B} (G')-1 = e_{\cup B} (G'')$ for any
                        $\{\{m,n\}\}\subset B \subset \sigma_{e(G')}$ and
                        $e_B(G_{e(G')}) = e_{\cup B} (G')=e_{\cup B} (G'')$ for any
                        $B \subset \sigma_{e(G')}\setminus \{\{m,n\}\}$. Therefore,
                                                \begin{align*}
        (q-1)^{2-|I|}\sum_{\pi
          \in\PPP(\sigma_{e(G')})} (-1)^{|\pi|-1}(|\pi|-1)! \prod_{B\in
                                                             \pi}q^{e_{\cup B} (G'')}
                                                             u_{\cup B} = \kumu^{(q)}_{\sigma_{e(G')}}(\uu(\sigma_{e(G')})^{G_{e(G')}}),
                                                \end{align*}
                                                which finishes the proof.
                                              \end{proof}

                                              \subsection{Macdonald and LLT cumulants}

Let $\nnu = (\la^1/\mu^1,\dots,\la^\ell/\mu^\ell)$ be an $\ell$-tuple of
skew Young diagrams. For any surjective function $f: [1..\ell] \to
[1..r]$, we say that a pair $(\nnu,f)$ is an \emph{$r$-colored
  tuple} of skew Young diagrams and we will think of it as an
$\ell$-tuple colored by $r$ colors, so that $i$-th element
$\la^i/\mu^i$ has color $f(i)$. For an $r$-colored
  tuple of skew Young diagrams $(\nnu,f)$ and for a subset $B
\subset [1..r]$, we define a tuple of
skew Young diagrams $(\nnu,f)^B$ as the sub-tuple of $\nnu$
colored by colors from $B$. More formally, $(\nnu,f)^B
:=(\la^{i_1}/\mu^{i_1},\dots,\la^{i_k}/\mu^{i_k})$, where $f^{-1}(B) =
\{i_1,\dots,i_k\}$ and $i_1 < \cdots < i_k$.

For a given $r$-colored
  tuple of skew Young diagrams $(\nnu,f)$, we define \emph{LLT cumulants} (with respect to different normalizations)
by the following formulae:
\begin{align}
  \label{eq:defLLTkumu}
	\ka_{\LLT^{\cospin}}(\nnu,f) := \kumu_{[1..r]}^{(q)}(\uu(\LLT^{\cospin})), \hspace{.2cm} &\text{where } \hspace{.2cm} u(\LLT^{\cospin})_B := \LLT^{\cospin}(\nnu,f)^B,\\
  	\ka_{\LLT^{\Mac}}(\nnu,f) := \kumu_{[1..r]}^{(q)}(\uu(\LLT^{\Mac})), \hspace{.2cm} &\text{where } \hspace{.2cm} u(\LLT^{\Mac})_B := \LLT^{\Mac}(\nnu,f)^B,\\
  	\ka_{\LLT}(\nnu,f) := \kumu_{[1..r]}^{(q)}(\uu(\LLT)), \hspace{.2cm} &\text{where } \hspace{.2cm} u(\LLT)_B := \LLT(\nnu,f)^B.
  \end{align}

Note that for any $\ell$-tuple of skew Young diagrams $\nnu$
there exists a unique $1$-colored tuple of skew Young diagrams
$(\nnu,\pi^{[1..\ell]}_{[1]})$, where $\pi^{[1..\ell]}_{[1]}$ is the unique
surjection of $[1..\ell]$ onto $\{1\}$. In this case, the cumulants
  $\ka_{\LLT^{\cospin}}(\nnu,\id_{[1..\ell]}),
  \ka_{\LLT^{\Mac}}(\nnu,\id_{[1..\ell]})$, and
  $\ka_{\LLT}(\nnu,\id_{[1..\ell]})$ coincide with the associated
  LLT functions $\LLT^{\cospin}(\nnu),
  \LLT^{\Mac}(\nnu)$, and
  $\LLT(\nnu)$, respectively. In general, LLT-cumulants can
  be interpreted as an $r$-colored generalization of LLT
  polynomials.

  The concept of $r$-colored
  tuples of skew shapes arose from the definition of cumulants of
  the symmetric functions naturally indexed by partitions. This definition was
  introduced in~\cite{DolegaFeray2017} (in the context of Jack
  and Macdonald symmetric functions) as follows: let
  $\{f_\lambda\}$ be a class of symmetric functions indexed by
  partitions. For partitions $\la^1,\dots,\la^r$,
we define the family $(\uu)$ indexed by subsets of $[1..r]$ as $u_B :=
f_{\lambda^B}$, where the partition $\lambda^B:=\bigoplus_{i
  \in B}\la^i$ is obtained from partitions $\la^i\colon i \in B$ by
summing their coordinates: $\la^B_j := \sum_{i \in B}\la^i_j$.
We observe that the data of partitions $\lambda^1,\dots,\lambda^r$ can
be alternatively encoded as an $r$-colored partition
$(\lambda = \lambda^{[1..r]},f)$ as follows: let $\lambda$ be a partition
and let $f\colon [1..\ell(\lambda)] \to [1..r]$ be a surjective function (that we interpret as the coloring of columns of the Young diagram
$\lambda$ by $r$ colors) such that the Young diagram formed by columns colored by $i$
is equal to $\lambda^i$. Then, it is clear that for every $B \subset [1..r]$, the Young
diagram formed by columns colored by colors in $B$ is equal to
$\lambda^B$. Of course, there are many colorings $f\colon [1..\ell(\lambda)]
\to [1..r]$ which encode partitions $\la^1,\dots,\la^r$ as an
$r$-colored partition $(\lambda,f)$, but among them there is a canonical
choice, which we call the \emph{canonical coloring} $(\lambda,f_{\cc}\colon [1..\ell(\lambda)]
\to [1..r])$. It is uniquely determined by the following property:
for any $i < j$ such that $\lambda'_i = \lambda'_j$ (we recall that $\lambda'$ denotes the \emph{transpose} of $\lambda$, i.e., the diagram with $\lambda_1$ boxes in the first column, $\lambda_2$ boxes in the second column, etc.), one has
$f_{\cc}(i) \leq f_{\cc}(j)$. This property simply means that the
Young diagram $\lambda^{[r]}$ can be obtained by sorting the columns
of $\lambda^1,\dots,\lambda^r$ such that all the columns of the same
height are ordered with respect to the natural order $1 < \cdots <r$,
see \cref{fig:ColoredPartitons}.

\begin{figure}
  \includegraphics[]{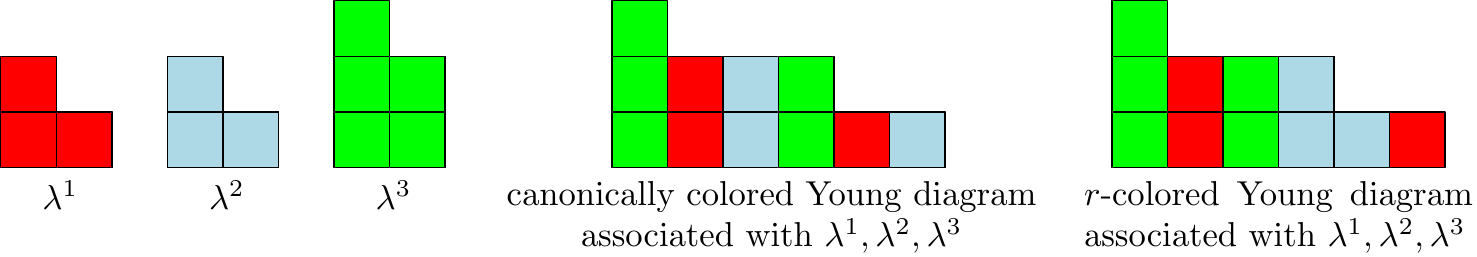}
  \caption{$r$-colored and canonically colored partitions.}
\label{fig:ColoredPartitons}
\end{figure}

When
$f_\lambda = \tilde{H}_\lambda^{(q,t)}$ is the
transformed version of the Macdonald polynomial indexed by a partition
$\lambda$, the corresponding $q$-partial cumulant
$\ka_{[1..r]}^{(q)}(\uu)$ is called the \emph{Macdonald cumulant}
and denoted $\ka(\la^1,\dots,\la^r)(\xx;q,t)$:
\begin{equation}
\label{eq:MacdoCumu}
\ka(\la^1,\dots,\la^r)(\xx;q,t) := \ka_{[1..r]}^{(q)}(\uu).
\end{equation}
It was studied in \cite{Dolega2017,Dolega2019}, where its
polynomiality and combinatorial interpretation was obtained,
generalizing the celebrated HHL formula~\eqref{eq:MacIntoLLT}. Furthermore,
it was conjectured in \cite{Dolega2017} that Macdonald cumulants are
Schur-positive:

  \begin{conjecture}[\cite{Dolega2017}]
    \label{conj:MacSchurPos}
Let $\lambda^1,\dots,\lambda^r$ be partitions. Then for any partition
$\nu$ the coefficient
$[s_\nu]\ka(\la^1,\dots,\la^r)$ in
the Schur expansion of the Macdonald cumulant is a polynomial in $q,t$ with
nonnegative coefficients.
\end{conjecture}

The first motivation for introducing LLT cumulants is to attack
\cref{conj:MacSchurPos}. Since Macdonald polynomials can be naturally
decomposed into LLT polynomials, it is natural to ask whether a similar
decomposition occurs for Macdonald cumulants. Moreover, it was
proved by Grojnowski and Haiman
\cite{GrojnowskiHaiman2007} that LLT polynomials are
Schur-positive, which gives an alternative proof of the Schur positivity of
Macdonald polynomials. Extensive
computer simulations performed by the authors using the SageMath computer algebra system
\cite{sagemath} tend us to believe that the result of Grojnowski
and Haiman might be a special case of Schur-positivity
that holds for the more general class of LLT cumulants. Therefore, we propose
the following conjecture:

  \begin{conjecture}
    \label{conj:LLTSchurPos}
For any $r$-colored tuple of skew shapes $(\nnu,f)$ and for any partition
$\la$ the coefficient
$[s_\la]\ka_{\LLT^{\cospin}}(\nnu,f)$ in
the Schur expansion of the LLT cumulant is a polynomial in $q$ with
nonnegative integer coefficients.
\end{conjecture}

\begin{example}
	Let $\nnu =
        ((2,2)/(1),(2),(1,1))$ be a tuple of three skew
        shapes. Consider two colorings $f\colon [1..3] \to [1..3]$ and $f'\colon
        [1..3] \to [1..2]$ defined by $f(i)=i$ for $1 \leq i \leq 3$
        and $f'(1) = f'(3) = 1, f'(2) = 2$. The corresponding cumulants
        $\ka_{\LLT^{\cospin}}(\nnu,f)$ and
        $\ka_{\LLT^{\cospin}}(\nnu,f')$ are equal to
        \begin{align*}
          \ka_{\LLT^{\cospin}}(\nnu,f) &= (q-1)^{-2}
                                         \bigg(\LLT^{\cospin}((2,2)/(1),(2),(1,1))-
                                         \LLT^{\cospin}((2,2)/(1))\cdot\\
                                         &\cdot\LLT^{\cospin}((2),(1,1))  - \LLT^{\cospin}((2))
                                         \LLT^{\cospin}((2,2)/(1),(1,1))
          \\ &-\LLT^{\cospin}((1,1))
                                         \LLT^{\cospin}((2,2)/(1),(2))\\
                                       &+2
          \LLT^{\cospin}((2,2)/(1))
          \LLT^{\cospin}((2))\LLT^{\cospin}((1,1))\bigg),\\
                    \ka_{\LLT^{\cospin}}(\nnu,f') &= (q-1)^{-1}
                                         \bigg(\LLT^{\cospin}((2,2)/(1),(2),(1,1))\\
                                       &-\LLT^{\cospin}((2))
                                         \LLT^{\cospin}((2,2)/(1),(1,1))\bigg).
          \end{align*}
        Expanding them in the basis of Schur functions we have
	\begin{align*}
		\ka_{\LLT^{\cospin}}(\nnu,f) &= (q^2+2q+2)s_{(2, 2, 1, 1, 1)} + (q+2)s_{(2, 2, 2, 1)} \\
		&+ (q^2+2q+2)s_{(3, 1,	1, 1, 1)} + (2q+4)s_{(3, 2, 1, 1)} \\
		&+ 2s_{(3, 2, 2)} + s_{(3, 3, 1)} +	(q+2)s_{(4, 1, 1, 1)} + s_{(4, 2, 1)},\\ \\
          \ka_{\LLT^{\cospin}}(\nnu,f') &= (q^3+q^2)s_{(2, 2, 1, 1, 1)} + (q^2+q)s_{(2, 2, 2, 1)} \\
          &+ (q^3+q^2)s_{(3, 1, 1, 1, 1)} + (2q^2+2q)s_{(3, 2, 1, 1)} + (2q+1)s_{(3, 2, 2)} \\
          &+ (q+1)s_{(3, 3, 1)} + (q^2+q)s_{(4, 1, 1, 1)} + (q+1)s_{(4, 2, 1)}.
	\end{align*}
\end{example}

In the following, we prove that \cref{conj:LLTSchurPos} implies
\cref{conj:MacSchurPos}. In order to do this, we express Macdonald
cumulants as a positive linear combination of LLT cumulants,
generalizing the classical decomposition from \cref{theo:HHL} to
cumulants, and we show that Schur-positivity of LLT
cumulants can be put into the following hierarchy: Schur-positivity of
$\ka_{\LLT^{\cospin}}(\nnu,f)$ implies Schur-positivity of
$\ka_{\LLT^{\Mac }}(\nnu,f)$, which further implies Schur-positivity of
$\ka_{\LLT}(\nnu,f)$.

\begin{remark}
	In fact, the chain of implications mentioned above is valid
        only when we restrict $\nnu$ to be a sequence of ribbon shapes due to the definition of the normalization $\LLT^{\Mac}(\nnu)$ (see \eqref{eq:LLT-Mac-def}). However, Schur-positivity of
	$\ka_{\LLT^{\cospin}}(\nnu,f)$ for all $r$-colored tuples of skew-shapes $(\nnu,f)$ implies Schur-positivity of
	$\ka_{\LLT}(\nnu,f)$ for all $r$-colored tuples of
        skew-shapes $(\nnu,f)$, which will be clear from the proof of \cref{theo:LLT->Mac}.
\end{remark}

\subsubsection{Decomposition of Macdonald cumulants}

\begin{theorem}
  \label{theo:MacCumuIntoLLTCumu}
    Let $\lambda^1,\dots,\lambda^r$ be partitions. Then, the following
    identity holds true:
    \begin{equation}
      \label{eq:MacCumuIntoLLT}
      \ka(\la^1,\dots,\la^r)(\xx;q,t) =
      \sum_{\nnu}t^{\maj(\nnu)}\ka_{\LLT^{\Mac}}(\nnu,f_{\cc}),
    \end{equation}
    where we sum over all tuples of ribbons whose bottom, far-right cell has content
    $0$ and such that $|\nnu_j| = (\la^{[1..r]})'_j$ for $1 \leq
    j \leq \ell(\la^{[1..r]})$ (i.e. the size
    of the $j$-th ribbon is equal to the length of the $j$-th column of
    $\lambda^{[1..r]}$) and $f_{\cc}$ is the canonical coloring
    associated with $\lambda^1,\dots,\lambda^r$.
  \end{theorem}

    \begin{proof}
      It is a direct consequence of the
      interpretation of the Macdonald cumulant as the $q$-partial cumulant of the
      canonically $r$-colored partition and of \cref{theo:HHL}. Indeed, note that for any subset $B
      \subset [1..r]$, \cref{theo:HHL} applied to $\lambda=\lambda^B$ gives
  \[          \tilde{H}_{\lambda^B}^{(q,t)} =
    \sum_{\nnu}t^{\maj(\nnu)}\LLT^{\Mac}(\nnu),\]
where we sum over skew-partitions whose $j$-th
    element is a ribbon of length $(\la^B)'_j$
    whose bottom, far-right cell has content
    $0$. In particular, for any set-partition $\pi \in \PPP([1..r])$,
    one has
      \[          \prod_{B \in \pi}\tilde{H}_{\lambda^B}^{(q,t)} =
    \sum_{\nnu}\prod_{B \in \pi}t^{\maj((\nnu,f_{\cc})^B)}\LLT^{\Mac}((\nnu,f_{\cc})^B),\]
    where the sum runs over the same set as the summation in \eqref{eq:MacIntoLLT}.

	Strictly from the definition~\eqref{eq:maj} of $\maj$, one
      has $\sum_{B \in \pi}\maj((\nnu,f_{\cc})^B) =
      \maj(\nnu)$ for any set-partition $\pi \in \PPP([1..r])$, and formula \eqref{eq:MacCumuIntoLLT} follows.
  \end{proof}

  \begin{theorem}
    \label{theo:LLT->Mac}
    Suppose that \cref{conj:LLTSchurPos} holds true. Then, for any
    $r$-colored tuple of skew shapes $(\nnu,f)$ and for any partition
$\nu$, the coefficients
\[ [s_\nu]\ka_{\LLT^{\Mac}}(\nnu,f)\in
  \Z_{\geq 0}[q], \qquad [s_\nu]\ka_{\LLT}(\nnu,f)\in
  \Z_{\geq 0}[q]\]
are polynomials
in $q$ with nonnegative integer coefficients. In particular,
\cref{conj:MacSchurPos} holds true.
\end{theorem}

\begin{proof}
Recall the definiton \cref{eq:defLLTkumu} of LLT cumulants. We will
show that there exist graphs $G \subset G'$ such that
\begin{align}
  \label{eq:pomocnicze1}
  \kumu_{[1..r]}^{(q)}(\uu(\LLT)) &=
  \kumu_{[1..r]}^{(q)}(\uu(\LLT^{\cospin})^{G'}),\\
  \kumu_{[1..r]}^{(q)}(\uu(\LLT^{\Mac})) &=
                                           \kumu_{[1..r]}^{(q)}(\uu(\LLT^{\cospin})^{G}).
                                           \label{eq:pomocnicze2}
\end{align}
Then the statements
  follow directly from \cref{theo:G-cumu} and \cref{eq:MacCumuIntoLLT}.

  Notice that the family of nonnegative
  integers $(e_B)_{B \subset V}$ indexed by subsets of the set $V$ is the
  number of edges in some graph $G = (V,E)$ linking vertices in $B$ if and only if
  \begin{equation}
    \label{eq:GraphCondition}
    e_B \geq \sum_{b \in B} e_{\{b\}}\ \ \ \text{ and }\ \ \ e_B = \sum_{\substack{B'\subset B,\\|B'| = 2}} e_{B'} - (|B|-2) \sum_{b
      \in B} e_{\{b\}}.
  \end{equation}
  Indeed, the inequality corresponds to $e_B$ counting all the loops on vertices from $B$, and the equality counts the edges between each pair of vertices from $B$ minus the overcounted loops.
  
  We first prove that there exist graphs $G,G'$ such that
  \eqref{eq:pomocnicze1} and \eqref{eq:pomocnicze2} hold. To show \eqref{eq:pomocnicze1}, consider $e_B = \min_{T \in \SSYT((\nnu,f)^B)} \inv(T) =
  |\Inv(T)\setminus \Inv_{\cospin}(T)|$, which does not depend on the
  choice of $T \in \SSYT((\nnu,f)^B)$.
  Then the conditions in \eqref{eq:GraphCondition} are satisfied
  since for a pair of boxes $\square \in
  (\nnu,f)^{\{i\}}$ and $\square' \in
  (\nnu,f)^{\{j\}}$, one has $(\square,\square') \in \Inv(T)\setminus
  \Inv_{\cospin}(T)$ if and only if $(\square,\square') \in \Inv(T_{\{i,j\}})\setminus
  \Inv_{\cospin}(T_{\{i,j\}})$, where $T_{\{i,j\}}$ is a tableau $T$
  restricted to $(\nnu,f)^{\{i,j\}}$.

  Similarly, for
  \[e_B = a((\nnu,f)^B) = \sum_{\square \in \Des((\nnu,f)^B)}|\{\square':
    c(\square')=c(\square), \tilde{c}(\square')>\tilde{c}(\square)\}|,\]
  one has
  \begin{multline}
    \label{eq:pomoc''}
    \sum_{\square \in \Des((\nnu,f)^B)}|\{\square'\colon
    c(\square')=c(\square), \tilde{c}(\square')>\tilde{c}(\square)\}| =
    \\  \sum_{b \in B}\sum_{\square \in \Des((\nnu,f)^{\{b\}})}\sum_{b'
      \in B}|\{\square' \in (\nnu,f)^{\{b'\}}\colon
    c(\square')=c(\square), \tilde{c}(\square')>\tilde{c}(\square)\}|.
  \end{multline}
  Note that for any subset $A \subset B$ and for each pair of boxes
  $(\square,\square') \in (\nnu,f)^{A}$, there is a uniquely associated
  pair of boxes $(\square,\square') \in (\nnu,f)^{B}$ and their
  contents are identical in $(\nnu,f)^{A} $ and $(\nnu,f)^{B}$, while
  their shifted contents might be different but the relation
  $\tilde{c}(\square')>\tilde{c}(\square)$ is again the same in both
  $(\nnu,f)^{A} $ and $(\nnu,f)^{B}$.
This observation together with \eqref{eq:pomoc''} implies that the
quantities $e_B$ satisfy \eqref{eq:GraphCondition}. This proves \eqref{eq:pomocnicze2}.

Finally, we prove that $G \subset G'$, which is equivalent to proving that $a((\nnu,f)^{B}) \leq \min_{T \in
  \SSYT((\nnu,f)^{B})} \inv(T)$. Let $\square \in
\Des((\nnu,f)^{B})$ and $\square' \in (\nnu,f)^{B}$ be such that
$c(\square')=c(\square), \tilde{c}(\square')>\tilde{c}(\square)$. For any
$T \in
  \SSYT((\nnu,f)^{B})$, we necessarily have $T(\square)>
  T(\square_{\downarrow})$. Therefore, either $(\square,\square') \in \Inv(T)$ or
  $(\square',\square_{\downarrow}) \in \Inv(T)$ (or both). Summing
  over all $\square \in
\Des((\nnu,f)^{B})$ and \[\square' \in \{\square' \in (\nnu,f)^{B}:
  c(\square')=c(\square), \tilde{c}(\square')>\tilde{c}(\square)\},\]
we
have that $a((\nnu,f)^{B}) \leq \inv(T)$ for any $T \in
  \SSYT((\nnu,f)^{B})$. Thus, $a((\nnu,f)^{B}) \leq \min_{T \in
  \SSYT((\nnu,f)^{B})} \inv(T)$, which is equivalent to the fact that
$G \subset G'$. This implies that $[s_\nu]\ka_{\LLT^{\Mac}}(\nnu,f)$ and $[s_\nu]\ka_{\LLT}(\nnu,f)$ are indeed polynomials in $q$, which finishes the proof.
\end{proof}

\section{Graph colorings}
\label{sec:GraphCol}

In the following, we interpret LLT polynomials as the generating
functions of colorings of certain directed graphs. This viewpoint
provides a natural graph-theoretic interpretation of LLT cumulants as
well as various positivity properties generalizing some recent results~\cite{AlexanderssonSulzgruber2022,Dolega2019}.



\subsection{LLT graphs and cumulants of $r$-colored LLT graphs}
\label{subsec:LLTGraphs}

\begin{definition} \label{def:LLTgraph}
	We call $G$ an \emph{LLT graph} if it is a finite directed graph with
        three types of edges, visually depicted as $\rightarrow$,
        $\twoheadrightarrow$, and $\Rightarrow$, which we call
        \emph{edges of type I}, \emph{of type II}, and \emph{double
          edges}, respectively. Denote the corresponding sets of edges
        by $E_1(G)$, $E_2(G)$, and $E_d(G)$. Additionally, write $\mathscr{G}$ for the $\Z[q]$-module spanned by LLT graphs and $\mathscr{G}_1 < \mathscr{G}$ for the
        submodule generated by LLT graphs with only edges of type II ($E_1(G) = E_d(G) = \emptyset$).
      \end{definition}

Let $\QSym$ denote the ring of quasi-symmetric functions over
$\mathbb{Z}[q]$. We recall that a quasisymmetric function $f$ is a
power series in variables $x_1,x_2,\dots$ of a bounded degree such
that for any sequence of positive integers $(\alpha_1,\dots,\alpha_n)$
the coefficients of the monomial $[x_{i_1}^{\alpha_1}\cdots
x_{i_n}^{\alpha_n}]f$ in $f$ is the same for all possible choices of
indices $i_1 < \dots < i_n$ (see~\cite{Gessel1984} for more details on
$\QSym$). With an LLT graph $G$ we associate its
\emph{LLT polynomial}:
	\begin{equation} \label{def:graphllt}
	\LLT(G):=\sum\limits_{f:V(G)\rightarrow\mathbb{N}} \left(\prod\limits_{(u,v)\in E(G)} \varphi_f(u,v)\right)\cdot \left(\prod\limits_{v\in V(G)} x_{f(v)}\right),
	\end{equation}
	with
	\begin{equation} \label{def:graphcol}
	\varphi_f(u,v)=\begin{cases} [f(u)>f(v)] & \text{for } (u,v)\in E_1(G); \\
	[f(u)\ge f(v)] & \text{for } (u,v)\in E_2(G); \\
	q[f(u)>f(v)]+[f(u)\le f(v)] & \text{for } (u,v)\in E_d(G), \end{cases}
	\end{equation}
	where $[A]$ is the characteristic function of condition $A$, i.e., is equal to $1$ if $A$ is true and $0$ otherwise.

There is an obvious way to associate an LLT graph $G_{\nnu}$ to a
sequence of skew shapes $\nnu$ such that $\LLT(G_{\nnu}) =
\LLT(\nnu)$. To be precise, vertices correspond to cells, edges of
type I go from a cell $\square$ to  $ \square_{\downarrow}$, edges of type II go from a cell $\square$ to $\square_{\leftarrow}$, and double edges connect cells that correspond to inversions (see \cref{fig:lltgraph}).

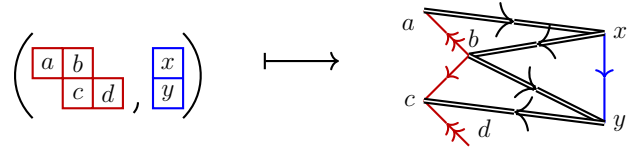
\begin{figure}
	\centering
	\begin{tikzpicture}[scale=.4, every node/.style={scale=0.8}]
	\draw[thick, red] (0,0) -- (-1,0) -- (-1,-1) -- (2,-1) -- (2,-2) -- (0,-2) -- (0,0) -- (1,0) -- (1,-2);
	\draw[thick, blue] (3,-1) -- (3,0) -- (4,0) -- (4,-2) -- (3,-2) -- (3,-1) -- (4,-1);
	\node at (-1,-3) {};
	\draw[thick] (2.5,-1.8) arc (0:-45:.5);
	\draw[thick] (-1,.5) arc (135:225:2);
	\draw[thick] (4,.5) arc (45:-45:2);
	
	\node at (-.5,-.5) {$a$};
	\node at (.5,-.5) {$b$};
	\node at (.5,-1.5) {$c$};
	\node at (1.5,-1.5) {$d$};
	\node at (3.5,-.5) {$x$};
	\node at (3.5,-1.5) {$y$};
	\end{tikzpicture}
	\begin{tikzpicture}[scale=1]
	\draw[thick, |->] (0,0) -- (1,0);
	\node at (-.5,-1) {};
	\end{tikzpicture}
	\hspace{.5cm}
	\begin{tikzpicture}[scale=0.6, every node/.style={scale=0.8}]
	\draw[thick, red] (0,0) -- (.5,-.5);
	\draw[thick, red, <<-] (.5,-.5) -- (1,-1);
	\draw[thick, red, ->] (1,-1) -- (.5,-1.5);
	\draw[thick, red] (.5,-1.5) -- (0,-2);
	\draw[thick, red] (0,-2) -- (.5,-2.5);
	\draw[thick, red, <<-] (.5,-2.5) -- (1,-3);
	
	\draw[thick, blue, ->] (4,-.5) -- (4,-1.5);
	\draw[thick, blue] (4,-1.5) -- (4,-2.5);
	
	\draw[thick, double, ->] (0,0) -- (2,-.25);
	\draw[thick, double] (2,-.25) -- (4,-.5);
	\draw[thick, double, ->] (4,-.5) -- (2.5,-.75);
	\draw[thick, double] (2.5,-.75) -- (1,-1);
	\draw[thick, double, ->] (1,-1) -- (2.5,-1.75);
	\draw[thick, double] (2.5,-1.75) -- (4,-2.5);
	\draw[thick, double, ->] (4,-2.5) -- (2,-2.25);
	\draw[thick, double] (2,-2.25) -- (0,-2);

	\node[below left] at (0,0) {$a$};
	\node[above] at (1.1,-1) {$b$};
	\node[left] at (0,-2) {$c$};
	\node[above right] at (1,-3) {$d$};
	\node[right] at (4,-.5) {$x$};
	\node[right] at (4,-2.5) {$y$};
	\end{tikzpicture}
	\caption{The LLT graph corresponding to ((3,2)/(1), (1,1)).}
	\label{fig:lltgraph}
      \end{figure}

 Let $G$ be an LLT graph and let $\vec{e_i} \in E_i(G)$ for $i \in \{1,d\}$. Define the local
 transformation
 \[\pi_{\vec{e_i}}(G) = \begin{cases} G\setminus\{\vec{e_1}\} -
    G_{\vec{e_1} \to \cev{e_2}}&\text{ for } i=1,\\ qG\setminus\{\vec{e_d}\} +
    (1-q)G_{\vec{e_d} \to \cev{e_2}}&\text{ for } i=d,\end{cases}\]
  where $\vec{e_i} \to \cev{e_j}$ ($\vec{e_i} \to \vec{e_j}$,
  respectively) denotes replacing the directed edge $\vec{e_i}$ of type $i$ by the
  edge of type $j$ with the opposite (the same, respectively)
  direction.
  
  \begin{example}
  	We have
  	\[
  	\pi_{(1,2)}\left(\vcenter{\hbox{
  			\begin{tikzpicture}[scale=.8]
  			\draw[thick,-{stealth}] (-.5,0) -- (-.125,.66);
  			\draw[thick] (-.125,.66) -- (0,.88);
  			\draw[thick,-{stealth}{stealth}] (0,.88) -- (.375,.22);
  			\draw[thick] (.375,.22) -- (.5,0);
  			\draw[thick,double] (0,0) -- (.5,0);
  			\draw[thick,double,-{stealth}] (-.5,0) -- (.16,0);
  			\node[below left] at (-.5,0) {$1$};
  			\node[above] at (0,.88) {$2$};
  			\node[below right] at (.5,0) {$3$};
  			\end{tikzpicture}}}\right)
  	=
  	\vcenter{\hbox{
  			\begin{tikzpicture}[scale=.8]
  			\draw[thick,-{stealth}{stealth}] (0,.88) -- (.375,.22);
  			\draw[thick] (.375,.22) -- (.5,0);
  			\draw[thick,double] (0,0) -- (.5,0);
  			\draw[thick,double,-{stealth}] (-.5,0) -- (.16,0);
  			\node[below left] at (-.5,0) {$1$};
  			\node[above] at (0,.88) {$2$};
  			\node[below right] at (.5,0) {$3$};
  			\end{tikzpicture}}}
  	-
  	\vcenter{\hbox{
  			\begin{tikzpicture}[scale=.8]
  			\draw[thick,-{stealth}{stealth}] (0,.88) -- (-.375,.22);
  			\draw[thick] (-.375,.22) -- (-.5,0);
  			\draw[thick,-{stealth}{stealth}] (0,.88) -- (.375,.22);
  			\draw[thick] (.375,.22) -- (.5,0);
  			\draw[thick,double] (0,0) -- (.5,0);
  			\draw[thick,double,-{stealth}] (-.5,0) -- (.16,0);
  			\node[below left] at (-.5,0) {$1$};
  			\node[above] at (0,.88) {$2$};
  			\node[below right] at (.5,0) {$3$};
  			\end{tikzpicture}}},\]
  	\[\pi_{(1,3)}\left(\vcenter{\hbox{
  			\begin{tikzpicture}[scale=.8]
  			\draw[thick,-{stealth}] (-.5,0) -- (-.125,.66);
  			\draw[thick] (-.125,.66) -- (0,.88);
  			\draw[thick,-{stealth}{stealth}] (0,.88) -- (.375,.22);
  			\draw[thick] (.375,.22) -- (.5,0);
  			\draw[thick,double] (0,0) -- (.5,0);
  			\draw[thick,double,-{stealth}] (-.5,0) -- (.16,0);
  			\node[below left] at (-.5,0) {$1$};
  			\node[above] at (0,.88) {$2$};
  			\node[below right] at (.5,0) {$3$};
  			\end{tikzpicture}}}\right)
  	=
  	q\vcenter{\hbox{
  			\begin{tikzpicture}[scale=.8]
  			\draw[thick,-{stealth}] (-.5,0) -- (-.125,.66);
  			\draw[thick] (-.125,.66) -- (0,.88);
  			\draw[thick,-{stealth}{stealth}] (0,.88) -- (.375,.22);
  			\draw[thick] (.375,.22) -- (.5,0);
  			\node[below left] at (-.5,0) {$1$};
  			\node[above] at (0,.88) {$2$};
  			\node[below right] at (.5,0) {$3$};
  			\end{tikzpicture}}}
  	+
  	(1-q)\vcenter{\hbox{
  			\begin{tikzpicture}[scale=.8]
  			\draw[thick,-{stealth}] (-.5,0) -- (-.125,.66);
  			\draw[thick] (-.125,.66) -- (0,.88);
  			\draw[thick,-{stealth}{stealth}] (0,.88) -- (.375,.22);
  			\draw[thick] (.375,.22) -- (.5,0);
  			\draw[thick,-{stealth}{stealth}] (.5,0) -- (-.16,0);
  			\draw[thick] (0,0) -- (-.5,0);
  			\node[below left] at (-.5,0) {$1$};
  			\node[above] at (0,.88) {$2$};
  			\node[below right] at (.5,0) {$3$};
  			\end{tikzpicture}}}.\]
  \end{example}
  
  We define
  \[ \pi(G) := \bigg(\prod_{\vec{e} \in E_1(G)\cup
      E_d(G)}\pi_{\vec{e}}\bigg)(G)\]
  as the concatenation of local transformations over all edges of type
  I and double edges (these transformations are commutative so their order does
  not matter and this concatenation is well-defined). Note that local transformations kill all edges of type I and $d$ and thus, the map  $\pi\colon \mathscr{G}\to \mathscr{G}_1$ is well-defined.
In fact, we claim that the map $\LLT\colon \mathscr{G} \to \QSym$ is a
well-defined surjective
homomorphism such that $\LLT(G) = \LLT\circ\pi(G)$ for every LLT graph.

\begin{lemma} \label{lem:arrowrel}
  For $\mathscr{G}$ and $\mathscr{G}_1$ as in \cref{def:LLTgraph}, the following diagram is commutative:
 \[
  \begin{tikzcd}
    \mathscr{G} \arrow{r}{\pi} \arrow[swap]{dr}{\LLT} & \mathscr{G}_1 \arrow{d}{\LLT} \\
     & \QSym
  \end{tikzcd}
\] 
\end{lemma}

\begin{proof}
  Let $G$ be an LLT graph. It was proved by F\'eray~\cite{Feray2015} that $\LLT\colon \mathscr{G}_1
  \to \QSym$ is a well-defined surjective homomorphism. Moreover, it
  is straightforward from the definition of the map $\LLT$ that it is
  invariant under the local
  transformations, i.e.,~for every $\vec{e} \in E_1(G)\cup
      E_d(G)$, one has $\LLT(\pi_{\vec{e}}(G)) = \LLT(G)$. Thus,
      $\LLT(G) = \LLT(\pi(G))$, which finishes the proof.	
      \end{proof}

      \begin{remark}
        \label{rem:LLTRel}
Let $\widehat{\mathscr{G}}_1<\mathscr{G}_1$ be a
       submodule of $\mathscr{G}_1 $ spanned by acyclic graphs. The main result of F\'eray~\cite{Feray2015} is an explicit description of the kernel of
        the map $\LLT\colon \widehat{\mathscr{G}}_1 \to \QSym$ by using the \emph{cyclic inclusion-exclusion principle}. This
        description together with \cref{lem:arrowrel} can be a priori
        used to describe the kernel of the morphism
        $\LLT\colon \mathscr{G} \to \QSym$, thus to understand all the
        relations between LLT graphs under the $\LLT$ morphism. Additionally, $\mathscr{G}$ seems
        to carry a natural Hopf algebra structure. Studying various
        relations between LLT polynomials is a very
       active topic
       recently and it proved to be useful in understanding the
       combinatorial structure of LLT
    polynomials~\cite{Lee2021,HuhNamYoo2020,AbreuNigro2021,AlexanderssonSulzgruber2022,Tom2021}. We
        believe that further studies in the direction of understanding
        the algebraic structure of the pair $(\mathscr{G},\LLT)$ might bring
        better understanding of the combinatorial structure of LLT
        polynomials, and we leave this problem for future research.
      \end{remark}

As a consequence of \cref{lem:arrowrel} and its proof, we obtain two identities
expressing the LLT polynomial of a given LLT graph $G$ in terms of two
important LLT graphs, which do not have any double edges. For any subset $E
\subset E_d(G)$, we define $G^E$ and $\tilde{G}^E$ as follows:
	\begin{itemize}
		\item $V(\tilde{G}^E) = V(G^E) = V(G)$,
		\item $E_d(\tilde{G}^E) = E_d(G^E) = \emptyset$,
		\item $E_1(\tilde{G}^E) = E_1(G^E) = E_1(G) \cup E$,
		\item $E_2(\tilde{G}^E) = E_2(G) \cup \{(u,v)\mid (v,u)\in E_d\setminus E\}$, and $E_2(G^E) = E_2(G)$.
	\end{itemize}

\begin{example}
	For
	\[G = 
	\vcenter{\hbox{
			\begin{tikzpicture}[scale=0.8]
			\draw[thick,-{stealth}] (0,0) -- (0,1);
			\draw[thick] (0,1) -- (0,2);
			\draw[thick,double,red] (2,0) -- (2,2);
			\draw[thick,double,red,-{stealth}] (2,0) -- (2,1);
			\draw[thick,double] (2,0) -- (0,2);
			\draw[thick,double,-{stealth}] (2,0) -- (1,1);
			\draw[thick,-{stealth}{stealth}] (0,0) -- (1.2,0);
			\draw[thick] (1.2,0) -- (2,0);
			\draw[thick,-{stealth}{stealth}] (2,2) -- (0.8,2);
			\draw[thick] (0.8,2) -- (0,2);
			\draw[thick,-{stealth}] (2,0) -- (3,0);
			\draw[thick] (3,0) -- (4,0);
			\draw[thick,double,red] (4,0) -- (2,2);
			\draw[thick,double,red,-{stealth}] (4,0) -- (3,1);
			\end{tikzpicture}}}\]
	and $E\subset E_d(G)$ equal to the set of the red edges above, we have
	\[G^E = 
	\vcenter{\hbox{
			\begin{tikzpicture}[scale=0.8]
			\draw[thick,-{stealth}] (0,0) -- (0,1);
			\draw[thick] (0,1) -- (0,2);
			\draw[thick] (2,0) -- (2,2);
			\draw[thick,-{stealth}] (2,0) -- (2,1);
			\draw[thick,-{stealth}{stealth}] (0,0) -- (1.2,0);
			\draw[thick] (1.2,0) -- (2,0);
			\draw[thick,-{stealth}{stealth}] (2,2) -- (0.8,2);
			\draw[thick] (0.8,2) -- (0,2);
			\draw[thick,-{stealth}] (2,0) -- (3,0);
			\draw[thick] (3,0) -- (4,0);
			\draw[thick] (4,0) -- (2,2);
			\draw[thick,-{stealth}] (4,0) -- (3,1);
			\end{tikzpicture}}},
	\hspace{2.5cm}
	\tilde{G}^E = 
	\vcenter{\hbox{
			\begin{tikzpicture}[scale=0.8]
			\draw[thick,-{stealth}] (0,0) -- (0,1);
			\draw[thick] (0,1) -- (0,2);
			\draw[thick] (2,0) -- (2,2);
			\draw[thick,-{stealth}] (2,0) -- (2,1);
			\draw[thick] (2,0) -- (0,2);
			\draw[thick,-{stealth}{stealth}] (0,2) -- (1.2,.8);
			\draw[thick,-{stealth}{stealth}] (0,0) -- (1.2,0);
			\draw[thick] (1.2,0) -- (2,0);
			\draw[thick,-{stealth}{stealth}] (2,2) -- (0.8,2);
			\draw[thick] (0.8,2) -- (0,2);
			\draw[thick,-{stealth}] (2,0) -- (3,0);
			\draw[thick] (3,0) -- (4,0);
			\draw[thick] (4,0) -- (2,2);
			\draw[thick,-{stealth}] (4,0) -- (3,1);
			\end{tikzpicture}}}.\]
\end{example}

\begin{corollary} \label{cor:LLTgraphnodoubleedges}
	For any LLT graph $G$, we have
	\[\LLT(G) = \sum_{E \subseteq E_d(G)} q^{|E|}\LLT(\tilde{G}^E) = \sum_{E \subseteq E_d(G)}(q-1)^{|E|}\LLT(G^E).\]
\end{corollary}
\begin{proof}
  Note that
  \[\bigg(\prod_{\vec{e} \in E_d(G)}\pi'_{\vec{e}}\bigg)(G) = \sum_{E
      \subseteq E_d(G)}q^{|E|}\tilde{G}^E, \ \ \bigg(\prod_{\vec{e} \in E_d(G)}\pi''_{\vec{e}}\bigg)(G) = \sum_{E
      \subseteq E_d(G)}(q-1)^{|E|}G^E,\]
  where $\pi'_{\vec{e_d}}(G) = G_{\vec{e_d}\to \cev{e_2}} +
    qG_{\vec{e_d} \to \vec{e_1}}$ and $\pi''_{\vec{e_d}}(G) = G\setminus\{\vec{e_d}\} +
    (q-1)G_{\vec{e_d} \to \vec{e_1}}$ for $e_d \in E_d(G)$. Moreover, $\LLT(G) = \LLT(\pi'_{\vec{e_d}}(G))$ follows from the definition, and $\LLT(G) = \LLT(\pi''_{\vec{e_d}}(G))$ follows from 
\[ \LLT(G) = \LLT(G_{\vec{e_d}\to \cev{e_2}}) + q\LLT(G_{\vec{e_d} \to
    \vec{e_1}}) = \LLT(G\setminus\{\vec{e_d}\}) +
    (q-1)\LLT(G_{\vec{e_d} \to \vec{e_1}})\]
since we have that $\LLT(G_{\vec{e_d}\to
  \cev{e_2}}) = \LLT(G\setminus\{\vec{e_d}\})-\LLT(G_{\vec{e_d}\to
  \vec{e_1}})$.
  \end{proof}

The definition of LLT cumulants of $r$-colored tuples of skew-shapes
generalizes naturally to the definition of LLT cumulants of $r$-colored LLT
graphs.

\begin{definition}
	We say that $(G,f)$ is an \emph{$r$-colored LLT graph} if $G$ is an LLT graph and $f\in V(G) \to [1..r]$ is a surjective coloring of vertices of $G$ such that both endpoints of edges in $E_1(G)\cup E_2(G)$ have the same color. For any subset $B \subset [1..r]$, we define the vertex set $V_B := \{v \in V(G)\colon f(v) \in B\}$ and for any subset $V' \subset V(G)$, we define $G|_V$ as the subgraph of $G$ obtained by restricting its set of vertices to $V'$. Then, we define the \emph{LLT cumulant} of an $r$-colored LLT graph $(G,f)$ as the $q$-partial cumulant $\ka^{(q)}_{[1..r]}(\uu)$ for the family defined by
	\[ u_B := \LLT(G|_{V_B}).\]
      \end{definition}

	Observe that the first equation in
        \cref{cor:LLTgraphnodoubleedges} is, in fact, a special case
        of a more general formula:
        \begin{corollary}
          \label{cor:LLTgraphnodoubleedgesPartition}
          For any set-partition $\pi \in \PPP([1..r])$, one has
	\begin{equation} \label{eq:LLTPartition}
		\prod_{B \in \pi}\LLT(G|_{V_B}) = \sum_{E \subseteq
		E_d(G)}\LLT(\tilde{G}^E)\prod_{B \in \pi}q^{|E_B|},
	\end{equation}
	where $E_B \subset E$ is the subset of edges with both endpoints in
	$B$.
      \end{corollary}
      \begin{proof}
        Formula \eqref{eq:LLTPartition} is proved similarly to
        \cref{cor:LLTgraphnodoubleedges}, so we only sketch the proof.
        Let $G_\pi := \bigoplus_{B\in\pi} G|_{V_B}$, where $G_1\oplus G_2$ is a disjoint union of the LLT graphs
        $G_1$ and $G_2$. Then $\prod_{B \in \pi}\LLT(G|_{V_B}) =
        \LLT(G_\pi)$. Note that $G_\pi$ is obtained from $G$ by removing all the double edges connecting vertices with colors lying in different
blocks of $\pi$. Consider two local transformations:
$\pi'_{\vec{e}}(G_\pi)$ for $\vec{e} \in E_d(G_\pi)$ and $\pi'''_{\vec{e}}(G_\pi) := (G_\pi)_{\vec{e} \to
  \vec{e_1}}+(G_\pi)_{\vec{e} \to \cev{e_2}}$ for any orientation $\vec{e}$ of $e
\notin E(G_\pi)$. Notice that
\[\bigg(\prod_{\vec{e} \in E_d(G_\pi) \atop \vec{\tilde{e}} \in
    E_d(G)\setminus E_d(G_\pi)}\pi'_{\vec{e}}\pi'''_{\vec{\tilde{e}}}\bigg)(G_\pi) = \sum_{E
    \subseteq E_d(G)}\tilde{G}^E \prod_{B \in \pi}q^{|E_B|}.\]
Finally, recall that $\LLT$ is invariant under taking the local
transformation $\pi'_{\vec{e}}$ and notice that $\LLT(G)=\LLT(\pi'''_{\vec{e}}(G))$ for any orientation $\vec{e}$ of $e
\notin E(G)$. This finishes the proof.
        \end{proof}

In the following, we prove that the LLT cumulant of the $r$-colored LLT
graph $(G,f)$ can be naturally expressed as a sum of LLT
polynomials of so-called $f$-connected graphs.

\begin{definition} \label{def:kaconnected}
	Let $(G,f)$ be an $r$-colored LLT graph. We say
        that it is
        \emph{$f$-connected} if the graph $G_f$
        obtained from $G$ by identifying vertices of the same color
        is connected. In other words, the graph $G$ is $f$-connected
        if for every pair $i,j \in [1..r]$, there exists
        $i=i_0\neq i_1 \cdots \neq i_k=j\in [1..r]$ and vertices
        $v_0,\dots,v_k \in V(G)$ colored by $i_0,\dots,i_k$
        respectively such that $v_{i-1}$ is connected to $v_i$ for
        every $1 \leq i \leq k$. 
      \end{definition}

      Note that when $f$ is a bijection
        then the graph $G$ is $f$-connected if and only if $G$ is
        connected, and if $f$ is a $1$-coloring then the condition of
        being $f$-connected is empty (it is always satisfied). We have the following
        combinatorial interpretation of an LLT cumulant of an $r$-colored LLT
        graph $(G,f)$.

\begin{theorem} \label{thm:cumasconnected}
	Let $(G,f)$ be an $r$-colored LLT graph and denote $E_d = E_d(G)$. Then:
	\begin{align}
		\ka_{\LLT}(G,f)(q+1) =
          \sum\limits_{\substack{E\subseteq E_d \\ G^{E} \text{
          $f$-connected}}} q^{|E|-r+1}\LLT(G^{E})(q+1).  \label{eq:cumasconnected}
\end{align}
\end{theorem}

\cref{thm:cumasconnected} essentially shows the structure behind the,
a priori, algebraic definition of a cumulant: it kills all
$f$-disconnected summands in the expansion and preserves the
$f$-connected ones. Furthermore, we note that we formulate the
statement with the polynomials evaluated at $q+1$ to highlight the
$\LLT$-positivity of the cumulant after the shift $q\longmapsto q+1$:
an operation that is also relevant in the context of the $e$-positivity phenomenon (see \cref{sec:epositivity}).





\begin{proof}[Proof of \cref{thm:cumasconnected}]
	We have
	\begin{align*}
\ka_{\LLT}(G,f)(q+1) :&= q^{1-r}\sum_{\pi\in\PPP([1..r])} (-1)^{|\pi|-1}(|\pi|-1)! \prod_{B\in\pi}\LLT(G|_{V_B}) \\
	&= q^{1-r}\sum_{\pi\in\PPP([1..r])} (-1)^{|\pi|-1}(|\pi|-1)!\LLT\left(G_\pi\right), \end{align*}
	where we recall that $G_\pi := \bigoplus_{B\in\pi} G|_{V_B}$.
	
	By \cref{cor:LLTgraphnodoubleedges}, for each $B\in\mathcal{P}([1..r])$, we get
	\begin{equation}
          \label{eq:kumu2}
          \LLT\left(G_\pi\right)(q+1) = \sum_{E \subseteq
            E_d(G_\pi)}q^{|E|}\LLT(G_\pi^E)(q+1)
        \end{equation}
        so that
        	\begin{align}
          \label{eq:kumu1}\ka_{\LLT}(G,f)(q+1) = q^{1-r}\sum_{\pi\in\PPP([1..r])} (-1)^{|\pi|-1}(|\pi|-1)! \sum_{E \subseteq
            E_d(G_\pi)}q^{|E|}\LLT(G_\pi^E)(q+1). \end{align}
	
	Now consider an LLT graph $G' = G_\sigma^E$ for some
        $\sigma\in\PPP([1..r])$ and $E\subseteq
        E_d(G_\sigma)$. For a fixed $G'$ of this form, pick $\sigma$ to be
        minimal, i.e.~ pick $\sigma$ such that for every block $B \in
        \sigma$, the graph $G'|_{V_B}$ is $f$-connected. Note that $G'$ is $f$-connected if
        and only if $\sigma = \{[1..r]\}$. We compute the
        contribution of the graph $G'$ to the RHS of the formula
        \eqref{eq:kumu1}.
	
	Note that $G'$ appears in a summand corresponding to a
        partition $\pi$ if and only if for every $B\in\sigma$, there
        exists $C\in\pi$ such that $B\subseteq C$. This is known as
        the containment relation $\sigma \leq \pi$ on the set of
        set-partitions. Therefore, we have
	\begin{align*}[\LLT(G')(q+1)]\ka_{\LLT}(G,f) &= q^{|E|-r+1}\sum_{\sigma\le\pi} (-1)^{|\pi|-1}(|\pi|-1)! = q^{|E|-r+1}\delta_{\sigma,\{[1..r]\}}. \end{align*}
	The last equality comes from the well-known fact that
        $(-1)^{|\pi|-1}(|\pi|-1)!$ is equal to the M\"{o}bius function
        $\mu(\pi,\{[1..r]\})$ on the poset of set-partitions
        $(\PPP([1..r]),\leq)$ and the sum of the
        M\"{o}bius function $\mu(\pi,\{[1..r]\})$ over the
        interval $\pi \in [\sigma,\{[1..r]\}]$ is non-zero (and equal to $1$)
        only if $\sigma = \{[1..r]\}$ (see, e.g., \cite{Weisner1935}). This finishes the proof as
        $\sigma = \{[1..r]\}$ if and only if $G'$ is $f$-connected.
      \end{proof}

      \subsection{Various positivity results}
      \label{subsec:Positivity}

      The purpose of this section is to derive various
      combinatorial formulae for an LLT cumulant of an $r$-colored
      LLT graph and proving certain positivity results. We start by a
      quick review on $G$-inversion polynomials and their different
      interpretations.

      \subsubsection{$G$-inversion polynomials and Tutte polynomials}
      \label{subsubsec:Tutte}

      Let $G$ be a multigraph (with possible multiedges and
      multiloops, as previously) on the set of vertices $[1..r]$. We
      say that $T$ is a \emph{spanning tree} of $G$ if it is a
      subgraph of $G$ with the same set of vertices $[1..r]$ and it is a tree (it is connected and has no cycles). A pair $(i,j)$ is called
      an \emph{inversion} of a spanning tree $T$ of $G$ if $i,j \neq 1$ and
      if $i$ is an ancestor of $j$ and $i>j$. An inversion $(i,j)$ is a \emph{$\ka$-inversion} if, additionally, $j$ is adjacent to the parent of $i$ in $G$. A $G$-inversion polynomial is a generating function of spanning trees of $G$ counted with respect to the number of $\ka$-inversions.

Let $\tilde{G}$ be a graph obtained
from $G$ by replacing all multiple edges by single ones.
We recall that for any subset $B \subset V$ we denote the number of
edges linking vertices in $B$ by $e_B$. The \emph{$G$-inversion
  polynomial} is given by
\begin{equation}
\label{eq:GesselSagan'}
\mathcal{I}_G(q) = q^{\text{number of loops in } G}\sum_{T \subset \tilde{G}}
q^{\ka(T)}\prod_{\{i,j\} \in T}[e_{\{i,j\}}(G)]_q,
\end{equation}
where the sum runs over all spanning trees of $\tilde{G}$, 
\begin{equation}
\label{eq:ka(T)}
\ka(T) = \sum_{\{i,j\} - \ka-\text{inversion in }
  T}e_{\{\text{parent}(i),j\}} (G),
\end{equation}
and we use the standard notation $[n]_q := \frac{q^n-1}{q-1} =
1+q+\cdots+q^{n-1}$. As we already mentioned in the introduction, $\mathcal{I}_G(q) =
\Tu(1,q)$, where $\Tu(x,y)$ is the Tutte polynomial of $G$ (a
classical graph invariant introduced by Tutte in~\cite{Tutte1954}):
\begin{equation}
\label{eq:DefTutte}
\Tu_G(x,y) = \sum_{H \subset G}(x-1)^{c(H)-1}\ (y-1)^{|E(H)|-|V|+c(H)}.
\end{equation}
The summation index above runs over all (possibly disconnected) subgraphs of $G$, $c(H)$ denotes the
number of connected components of $H$, and $E(H)$ is the set of edges
of $H$. In fact, we have the following lemma, which is essentially due
to Gessel~\cite{Gessel1995} and Josuat-Vergès~\cite{Josuat-Verges2013} (see
also~\cite{Dolega2019} for treating both frameworks in the setting of multigraphs).

\begin{lemma}
  \label{lem:Tut}
  Let $G$ be a multigraph with the vertex set $V = [1..r]$ and let
  $\uu$ be a family indexed by subsets of $[1..r]$ defined as $u_B
  := q^{e_B}$ for every $B \subset [1..r]$. Then we have the following
  equalities between the generating series:
  \begin{equation}
  \mathcal{I}_G(q)  = \Tu_G(1,q) = \ka^{(q)}(\uu).
    \end{equation}
  \end{lemma}

      \subsubsection{Monomial positivity}
      \label{subsubsec:monomial}

      Here, we prove the following theorem implying positivity of LLT
      cumulants for arbitrary $r$-colored LLT graphs in the
      quasi-symmetric monomial basis (this is a refinement of the main result
      from~\cite{Dolega2019}):

  \begin{theorem} \label{thm:kumuMonoPos}
	Let $(G,f)$ be an $r$-colored LLT graph and denote $E_d = E_d(G)$. Then:
	\begin{align}
          \ka_{\LLT}(G,f)(q) =
          \sum\limits_{\substack{E\subseteq E_d \\ \hat{G}^{E} \text{
          $f$-connected}}} \I_{(\hat{G}^E)_f}(q)\LLT(\tilde{G}^{E})(q),  \label{eq:Tutte}
\end{align}
	where $\hat{G}^{E}$ is obtained from $\tilde{G}^E$ by removing
        all the edges of type II (i.e. $E(\hat{G}^{E}) =
        E(\tilde{G}^{E})\setminus E_2(\tilde{G}^E)$).
      \end{theorem}

      \begin{proof}

Let $\hat{G}^{E}$ be a graph obtained from $\tilde{G}^E$ by removing
        all the edges of type II and we recall that $(\hat{G}^{E})_f$
        is a graph obtained from $\hat{G}^{E}$ by identifying vertices
        of the same color, i.e. $v\sim w$ if $f(v) = f(w)$. Note that the
        vertex set of $(\hat{G}^{E})_f$ is equal to $[1..r]$ and
        $|E_B|$ in the previous formula is equal to the number of edges
        of $(\hat{G}^{E})_f$ with both endpoints belonging to $B$
        (that we denote by $e_B$ to be consistent with the previous notation).
Therefore, following \eqref{eq:LLTPartition}, we end up with the formula
\begin{align*}\ka_{\LLT}(G,f) &= (q-1)^{1-r}\sum_{E \subseteq
    E_d(G)}\LLT(\tilde{G}^E)\bigg(\sum_{\pi \in
    \PPP([1..r])}(-1)^{|\pi|-1}(|\pi|-1)!\prod_{B \in \pi}
                                q^{|E_B|}\bigg) \\
  &= \sum_{E \subseteq
  E_d(G)}\ka^{(q)}_{[1..r]}(\uu)\LLT(\tilde{G}^E),
  \end{align*}
where $u_B := q^{e_B}$. This can be rewritten as
\[ \sum_{E \subseteq
    E_d(G)}\I_{(\hat{G}^E)_f}(q)\LLT(\tilde{G}^{E})(q)\]
thanks to \cref{lem:Tut}. Finally, $\I_{(\hat{G}^E)_f}(q)=0$ whenever
$(\hat{G}^E)_f$ is not connected (because disconnected graphs have no spanning trees), which is
the very definition of being $f$-connected for $\hat{G}^{E}$. This finishes
the proof.
\end{proof}

\subsubsection{Fundamental quasisymmetric functions and
  \cref{conj:LLTSchurPos} for hooks}
\label{subsub:Fundamental}
    
    For any non-negative integer $n$ and a subset $A\subset [n-1]$, we define the \emph{fundamental quasisymmetric function} $F_{n,A}(\xx)$ to be the expression
    \[F_{n,A}(\xx) := \sum_{\substack{i_1\le\dots\le i_n \\ j\in A\Longrightarrow i_j < i_{j+1}}} x_{i_1}\dots x_{i_n}.\]
    
    We say that a tableau $T\in\SSYT(\nnu)$ of a
    sequence $\nnu$ with $|\nnu|=n$ is \emph{standard} if
    $T:\nnu \rightarrow [n]$ is a bijection, and denote that fact
    by $T\in\SYT(\nnu)$. We also define the set of
    descents $\Des(T)$ of $T$ (note that this is not the same as the
    set of descents of a tuple of skew shapes $\nnu$, which
    appeared in the definition of Macdonald polynomials) as the set of $i\in[1..n]$ such that
    $\tilde{c}(T^{-1}(i+1)) < \tilde{c}(T^{-1}(i))$.

    \vspace{5pt}

    In \cite{HaglundHaimanLoehr2005}, Haglund, Haiman and Loehr
    implicitly\footnote{instead of LLT polynomials they expanded Macdonald
      polynomials into fundamental quasisymmetric functions, but their arguments can be directly
      applied to LLT polynomials yielding \eqref{eq:lltinquasisymmetric}} proved
    the following formula for the expansion of LLT polynomials in the fundamental quasisymmetric functions.
    
    \begin{theorem}[\cite{HaglundHaimanLoehr2005}]
    	For a sequence of skew shapes $\nnu$ with $|\nnu| = n$, we have
    	\begin{equation} \label{eq:lltinquasisymmetric}
    	\LLT(\nnu) = \sum\limits_{T\in\SYT(\nnu)}
        q^{\inv(T)} F_{n,\Des(T)}(\xx).
    	\end{equation}
      \end{theorem}
  
  What is more, we can obtain a similar result in our language and notation.
  
  	\begin{corollary}
  		For any $r$-colored tuple $(\nnu,f)$ of size $n$ and for any set partition $\pi\in\PPP([1..r])$, we have
  		\begin{equation} \label{eq:prodoflltinquasisymmetric}
  		\prod_{B \in \pi}\LLT((\nnu,f)^B) = \sum\limits_{T\in\SYT(\nnu)}
  		q^{\inv_{\pi}(T)} F_{n,\Des(T)}(\xx),
  		\end{equation}
  		where $\inv_{\pi}(T)$ denotes the number of inversions in $T$
  		with both boxes in the same block of $\pi$.
  	\end{corollary}
    \begin{proof}
    	The result is a straightforward application of the arguments used in ~\cite{HaglundHaimanLoehr2005}.
    \end{proof}
      
      Applying the same proof as in \cref{thm:kumuMonoPos} to \eqref{eq:prodoflltinquasisymmetric}, we obtain the following result (see also
      \cite[Section 5]{Dolega2019} for an analogous argument applied
      to Macdonald cumulants):
      
    \begin{theorem}
      \label{theo:LLTFundamental}
    	Let $(\nnu,f)$ be an $r$-colored sequence of skew shapes
        of size $n$. Then:
    	\begin{equation} \label{eq:lltcumulantinquasisymmetric}
    	\ka_{\LLT}(\nnu,f)(q) = \sum\limits_{T \in \SYT(\nnu)} \sum \I_{(\widehat{G_{\nnu}}^{E^T})_f}(q)F_{n,\Des(T)}(\xx),
    	\end{equation}
    	where the second sum runs over all subsets $E^T\subseteq E_d(G_{\nnu})$ for which $\widehat{G_{\nnu}}^{E^T}$ is $f$-connected and $T(i) > T(j)$ whenever $(i,j)\in E(\widehat{G_{\nnu}}^{E^T})$.
    \end{theorem}
    
    In \cite{Dolega2019}, we were able to find an explicit formula for
    the coefficients of Schur symmetric functions indexed by
    \emph{hooks}, i.e.,~ partitions of the form $(k,1^{n-k})$, in
    Macdonald cumulants, thanks to the arguments
    from~\cite{HaglundHaimanLoehr2005}. Here, we will use a very nice
    theorem of Egge, Loehr and
    Warrington~\cite{EggeLoehrWarrington2010} which gives a
    combinatorial description of Schur coefficients of any symmetric
    function when given an expansion in fundamental quasisymmetric
    functions.

    \begin{theorem}[\cite{EggeLoehrWarrington2010}] \label{thm:quasitoschurs}
    	Suppose that
    	\[\sum\limits_{\la\vdash n} c_\la s_\la =
          \sum\limits_{\alpha\models n} d_\alpha F_{n,A(\alpha)},\]
        where $A(\alpha) = (\alpha_1,\alpha_1+\alpha_2,\dots,\sum_{i=1}^{\ell(\alpha)-1}\alpha_i)$.
    	Then we have $c_{(k,1^{n-k})} = d_{(k,1^{n-k})}$ for all $1\le k\le n$.
    \end{theorem}
    
The original result from~\cite{EggeLoehrWarrington2010} gives a
description of the coefficients $c_\la$ for a general $\la\vdash
n$. However, since we only need the case in the statement (i.e., when
$\la$ is a hook), we refer interested readers
to~\cite{EggeLoehrWarrington2010} for the general version, which is
slightly more complicated.

The following theorem is an immediate corollary of \cref{theo:LLTFundamental} and \cref{thm:quasitoschurs}:
    
    \begin{theorem}
      \label{theo:LLTShurHooks}
    	Let $(\nnu,f)$ be an $r$-colored sequence of skew shapes
        of size $n$. Then for any $1 \leq k \leq n$
    	\[[s_{(k,1^{n-k})}]\ka_{\LLT}(\lla) = \sum\limits_{\substack{T\in\SYT(\lla) \\ \Des(T) = \{k,k+1,\dots,n-1\}}} \sum \I_{(\widehat{G_{\nnu}}^{E^T})_f}(q),\]
    	where the second sum runs over all subsets $E^T\subseteq E_d(G_{\nnu})$ for which $\widehat{G_{\nnu}}^{E^T}$ is $f$-connected and $T(i) > T(j)$ whenever $(i,j)\in E(\widehat{G_{\nnu}}^{E^T})$.
    \end{theorem}

      \subsubsection{e-positivity} \label{sec:epositivity}

      Let $(e_\lambda)$ be the basis of elementary symmetric functions,
      i.e.~ $e_\lambda := \prod_{i=1}^{\ell(\lambda)}e_{\lambda_i}$,
      where $e_{i} := \sum_{j_1 < \cdots <j_i}x_{j_1}\cdots x_{j_i}$ is the $i$-th \emph{elementary
        symmetric function}. $e$-positivity of a given symmetric
      function $f$ is a stronger property than Schur-positivity and it
      suggests a specific interpretation of the function $f$ in terms
      of the representation
      theory of the symmetric group, and in algebro-geometric context. This observation recently generated a
      lot of research in studying $e$-positive symmetric functions,
      and after a series of conjectures
      \cite{Bergeron2017,AlexanderssonPanova2018,GarsiaHaglundQiuRomero2019},
      it was clear that $e$-positivity of a big class of symmetric
      functions would be a consequence of $e$-positivity for
      \emph{vertical-strip LLT polynomials} after the shift $q \to q+1$, i.e. for $\LLT(\nnu)(q+1)$ where
      $(\nnu)_i = (1^{n_i+k_i})/(1^{k_i})$ for each $1 \leq i \leq
      \ell(\nnu)$ and some nonnegative integers $n_i,k_i$. An explicit combinatorial formula for the
      coefficients of vertical-strip LLT polynomials in the basis of
      elementary functions was independently conjectured in
      \cite{GarsiaHaglundQiuRomero2019,Alexandersson2021}\footnote{in
        fact, these interpretations are not identical, since the
        authors use slightly different framework in their works, but
        it is possible to show that they
        are equivalent} and shortly afterwards the positivity (without
      proving the combinatorial interpretation) was proved in
      \cite{DAdderio2020} and subsequently \cite{AlexanderssonSulzgruber2022}
      finalized the picture by proving the combinatorial
      interpretation. In the following, we reformulate this
      combinatorial interpretation in our current framework.

      Let $\nnu$ be a tuple of vertical-strips and let $G =
      G_{\nnu}$ be the associated LLT-graph. We recall that a vertex
      $v \in V(G)$ is associated with a box $\square(v) \in \nnu$ and the vertices are
      naturally labeled by the shifted contents of the corresponding
      boxes $\tilde{c}(v) := \tilde{c}(\square(v))$. Fix $E \subset
      E_d(G)$ and define $G' = G^E$. Since $G'$ is a directed graph
      (note that the condition that $\nnu$ is a tuple of vertical
      strips implies that $G'$ has only edges of type I),
      some of the vertices of $G'$ have only outgoing edges -- such
      vertices are called \emph{sources}. We define the following
      equivalence relation on the set of vertices $V(G')$: the
      vertices $v\sim w$ are in the same equivalence class if the source
      $\theta(v)$ with the smallest label from which there exists a
      directed path to $v$ is the same as the source
      $\theta(w)$ with the smallest label from which there exists a
      directed path to $w$. The partition $\lambda(G')$ is defined
      as the partition whose parts are sizes of the equivalence
      classes in this relation. See \cref{fig:eGraphs} for an
      example.

      \begin{figure}
  \includegraphics[width=\linewidth]{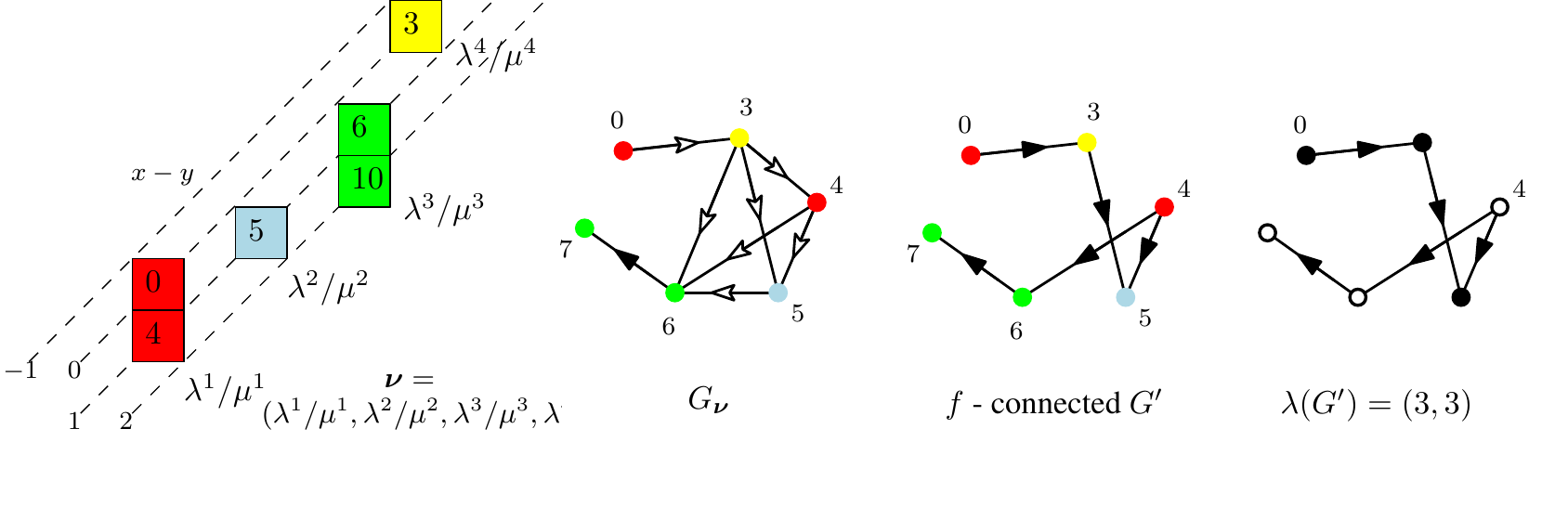}
  \caption{A tuple $\nnu$ of vertical strips and the
    associated LLT graph $G_{\nnu}$. The labels of vertices are
    the shifted contents of the corresponding boxes. Graph $G' = G^E$
  for $E = \{(0,3),(3,5),(4,5),(4,6)\}$ is $f$-connected, where
  $f(i)=i$. In the last picture, the displayed labels are the labels of
the sources of $G'$ and the corresponding two equivalence classes
$\{0,3,5\}$ and $\{4,6,7\}$ are
depicted by the whole and the empty vertices, respectively.}
\label{fig:eGraphs}
\end{figure}

\begin{theorem}{\cite{AlexanderssonSulzgruber2022}}
  \label{theo:AS}
        Let $\nnu$ be a tuple of vertical-strips and let $G =
      G_{\nnu}$ be the associated LLT-graph. Then 
      \begin{equation}
        \label{eq:AS}
        \LLT(\nnu)(q+1) = \sum_{E \subseteq
          E_d(G)}q^{|E|}e_{\lambda(G^E)}.
        \end{equation}
      \end{theorem}

      In the following, we show that the vertical-strip LLT
      cumulants preserve $e$-positivity, which refines \cref{theo:AS},
      but, most importantly, shows that $e$-positivity of
      vertical-strip LLT polynomials
      naturally decomposes into $f$-connected components, each
      corresponding to the vertical-strip LLT cumulant. In other
      terms, heuristically, the e-positivity of vertical-strip LLT
      polynomials is ``built'' from e-positivity of LLT cumulants, which
      naturally decompose LLT polynomials from the graph-coloring
      point of view.

\begin{theorem}
  \label{theo:e-positivity}
        Let $(\nnu,f)$ be an $r$-colored tuple of vertical-strips and let $G =
      G_{\nnu}$ be the associated LLT-graph. Then 
      \begin{equation}
        \label{eq:KumuInE}
        \ka_{\LLT}(\nnu,f) (q+1) = \sum\limits_{\substack{E\subseteq E_d \\ G^E \text{ $f$-connected}}} q^{|E|+1-r}e_{\lambda(G^E)}.
        \end{equation}
      \end{theorem}

      \begin{proof}
        We recall that the $q$-partial cumulant of the family $(\uu)$
        is defined by the formula \eqref{def:QCumu}. One can invert
        this formula in order to express $u_I$ in terms of the
        $q$-partial cumulants:
        \[ u_I = \sum_{\pi \in \PPP(I)}(q-1)^{|I|-|\pi|}\prod_{B \in
            \pi}\ka_B^{(q)}(\uu).\]
        Applying this to our setting, we obtain that for any $r \geq
        1$ and for any $r$-colored tuple $(\nnu,f)$, one has
        \[ \LLT(\nnu)(q+1) = \sum_{\pi \in \PPP([1..r])}q^{r-|\pi|}\prod_{B \in
            \pi}\ka_{\LLT}((\nnu,f)^B,f|_B)(q+1),\]
        where $f|_B$ is the $|B|$-coloring of
        $(\nnu,f)^B$ obtained from $f$ by restricting it to the
        preimage of $B$, i.e., $f|_B\colon f^{-1}(B) \to B$.
        
        We prove \eqref{eq:KumuInE} by induction
        on $r$. For $r=1$, the LHS of \eqref{eq:KumuInE} is equal to
        $\LLT(\nnu)(q+1)$, while the RHS of \eqref{eq:KumuInE}
        coincides with the RHS of \eqref{eq:AS}, because every
        $1$-colored graph is trivially $f$-connected. Let
        $(\nnu,f)$ be an $r$-colored tuple of vertical-strips
        with $r >1$. Let $G' = G^{E}$ for some $E \subseteq
          E_d(G)$. Note that decomposing $G'$ into $f$-connected
          components, we find a set-partition $\pi \in \PPP([1..r])$
          such that each $f$-connected component has a vertex set $V_B
          := \{v \in V(G') \text{ colored by } b \in B\}$ for some $B
          \in \pi$. Therefore, we can rewrite \eqref{eq:AS} as follows

\[\LLT(\nnu)(q+1) = \sum_{\pi \in \PPP([1..r])}\prod_{B \in
    \pi}\bigg(\sum\limits_{\substack{E_B\subseteq E_d(G_{B}) \\
      G_B^{E_B} \text{ $f|_B$-connected}}}
  q^{|E_B|}\bigg)e_{\lambda(\bigoplus_B G_{B}^{E_B})}.\]

      Notice also that $e_{\lambda(\bigoplus_BG|_{V_B}^{E_B})} =
      \prod_{B \in \pi}e_{\lambda(G|_{V_B}^{E_B})}$, which is
      immediate from the definition of $\lambda(G')$. Indeed, the
      whole equivalence class has to be contained in the connected
      component of $G$, which is further contained in the
      $f$-connected component. Using the obvious identity
      \[ q^{r-|\pi|} = \prod_{B \in \pi}q^{|B|-1},\]
      we obtain
      \begin{align*} &\ka_{\LLT}(\nnu,f)(q+1) = \sum_{\pi \in \PPP([1..r])}q^{1-|\pi|}\prod_{B
          \in \pi}\bigg(\sum\limits_{\substack{E_B\subseteq E_d(G_{B})\colon \\
      G_B^{E_B} \text{ is $f|_B$-connected}}}
  q^{|E_B|}\bigg)e_{\lambda(\bigoplus_B G_{B}^{E_B})}\\
        -&\sum\limits_{\substack{\pi \in \PPP([1..r]) \\ \pi \neq \{[1..r]\}}} q^{r-|\pi|}\prod_{B \in
        \pi}\ka_{\LLT}((\nnu,f)^B,f|_B)(q+1) 
        = \sum\limits_{\substack{E\subseteq E_d \\ G^{E} \text{ $f$-connected}}} q^{|E|+1-r}e_{\lambda(G^E)},
      \end{align*}
      where the last equality follows from the inductive
      hypothesis, and the proof is finished.
        \end{proof}

        \section{Concluding remarks and questions}
        \label{sec:Coclude}

        We conclude by proving \cref{conj:LLTSchurPos} for some special cases and stating some more general open questions.
        
        We start by showing that \cref{conj:LLTSchurPos} holds true when
        $\ell(\nnu) = 2$.

        \begin{proposition}
          Let $\nnu = ((\la^1/\mu^1,\la^2/\mu^2),f)$ be an
          $r$-colored pair of
          skew Young diagrams. Then, for every partition $\la$ the coefficient
          \[ [s_\la]\ka_{\LLT^{\cospin}}(\nnu,f) \in
            \mathbb{Z}_{\ge 0}[q]\]
          is a polynomial in $q$ with nonnegative integer
          coefficients.
        \end{proposition}

        \begin{proof}
          We know that LLT polynomials are Schur positive, i.e
          \[ \LLT^{\cospin}(\nnu)(q) = \sum_\la c_{\la^1/\mu^1,
              \la^2/\mu^2}^\la(q)s_\la,\]
          where $c_{\la^1/\mu^1,
            \la^2/\mu^2}^\la(q)  = \sum_{i=0}^{d_{\la^1/\mu^1,
            \la^2/\mu^2}^\la}c_{\la^1/\mu^1,
            \la^2/\mu^2}^{\la;i}q^i\in \mathbb{Z}_{\ge 0}[q]$ and we know that
        \[ \LLT^{\cospin}(\la^1/\mu^1)(q)
          \LLT^{\cospin}(\la^2/\mu^2)(q) =
          \LLT^{\cospin}(\nnu)(1). \]
        Therefore, the case of $2$-coloring gives us
        \[ \ka_{\LLT^{\cospin}}(\nnu,f) = \sum_\la \sum_{i=1}^{d_{\la^1/\mu^1,
              \la^2/\mu^2}^\la}c_{\la^1/\mu^1,
            \la^2/\mu^2}^{\la;i}[i]_qs_\la.\]
        Since the LLT cumulant of $1$-colored tuple is simply an LLT
        polynomial (which is Schur positive by the result of
        Grojnowski and Haiman~\cite{GrojnowskiHaiman2007}) and there
        are no other $r$-colorings of a pair of skew partitions, the
        proof is finished.
      \end{proof}

      \begin{remark}
        Note that in this simple case, the coefficient $[s_\la]\ka_{\LLT^{\cospin}}(\nnu,f)$
        is explicit assuming that the coefficient
        $[s_\la]\LLT^{\cospin}(\nnu)$ is known. In our setting,
        this coefficient was described combinatorially in terms of
        inversions of Yamanouchi tableaux by Roberts~\cite{Roberts2014},
        which, in effect, provides also the combinatorial interpretation
        of the coefficient $[s_\la]\ka_{\LLT^{\cospin}}(\nnu,f)$.
      \end{remark}

      An explicit expression for $\LLT(\nnu)$ in the Schur basis
      exists also for $\ell(\nnu)=3$ due to
      Blasiak~\cite{Blasiak2016} but it is much more complicated and,
      as noticed by Blasiak, there are serious
      difficulties in going beyond the case $\ell(\nnu)=
      3$. Let us recover Blasiak’s result here \cite[Corollary 4.3]{Blasiak2016}, so that we can state our conjecture connected to its cumulant counterpart.
      
      Let $\nnu = (\la^1/\mu^1, \la^2/\mu^2, \la^3/\mu^3)$. Blasiak proved that

      \begin{equation} \label{eq:LLTof3}
\LLT(\nnu)(q) = \sum_\la c_{\nnu}^\la(q)s_\la, \hspace{1.5cm} \text{where} \hspace{1.5cm}
c_{\nnu}^\la(q) = \sum_{\substack{T\in\RSST(\la) \\ \Des_3'(T) = D'(\nnu) \\ \tilde{c}(\nnu) \text{ - entries of } T}} q^{\inv_3'(T)},
\end{equation}
and
\begin{itemize}
	\item $\RSST(\la)$ is the set of \emph{restricted square strict tableaux} of shape $\la$, i.e., fillings of $\la$ whose columns strictly increase upwards, rows strictly increase rightwards, and the filling of the cell $(x,y)$ is smaller by at least $3$ than that of $(x',y')$ with $x'>x$ and $y'>y$;
	\item $\Des'_3(T)$ is the multiset of pairs $(T(x,y),T(x',y'))$ with $(x,y),(x',y')\in\sh(T)=\lambda$, such that $T(x,y) - T(x',y') = 3$, and either $y > y'$ and $x \le x'$, or $x = x' +1$, $y = y' + 1$, and $T(x',y) = T(x,y) - 1$;
	\item $D'(\nnu)$ is the multiset of pairs $(\tilde{c}(x,y),\tilde{c}(x',y'))$ with $(x,y),(x',y')\in\nnu$, such that $\tilde{c}(x,y) = \tilde{c}(x',y') + 3$ and $y < y'$ and $x \le x'$;
	\item $\tilde{c}(\nnu)$ is the sequence of shifted contents of $\nnu$; and
	\item $\inv'_3(T)$ is the number of pairs $((x,y),(x',y'))$ with $(x,y),(x',y')\in\sh(T)$ with $0 < T(x,y) - T(x',y') < 3$, such that $y > y'$ and $x \le x'$.
\end{itemize}

Note that the sets $\Des_3'(T)$ and $D'(\nnu)$ are indeed multisets. For instance, for $\nnu = ((3,3,3),(1),(1))$, we have
\[D'(\nnu) =
  \{(6,3),(3,0),(3,0),(3,0),(0,-3),(0,-3),(0,-3),(-3,-6)\}.\]
The point $(3,0)$ in $D'(\nnu)$ counted with multiplicity $3$ comes
from the following pairs $(x,y),(x',y')\in\nnu$: $(2,1),(2,2) \in
\nnu_1$, $(2,1),(3,3) \in
\nnu_1$, and $(3,2),(3,3) \in
\nnu_1$.

\begin{example}
Let $\lambda^1/\mu^1 = \lambda^3/\mu^3 = (1,1)$ and $\lambda^2/\mu^2 =
(2,2)/(2)$ and consider $[s_{(3,2,1)}]\LLT(\nnu)$ for $\nnu
= (\lambda^1/\mu^1, \lambda^2/\mu^2, \lambda^3/\mu^3)$. According to
\eqref{eq:LLTof3}, it is counted by restricted square
strict tableaux of shape $(3,2,1)$ with some additional
constraints. On the left hand side of \cref{fig:Blasiak}, we show
$\nnu$ with its shifted contents and we give an example of a
restricted square
strict tableau $T$ of shape $(3,2,1)$, which satisfies the constraint
$\Des_3'(T) = D'(\nnu)$. We colored the boxes of $\la^i/\mu^i$,
therefore the pairs counting $\inv'_3(T)$ can be represented as the
edges of a graph on three vertices (which is shown on two drawings on
the right hand side).
\begin{figure}[h!]
  \includegraphics[width=\linewidth]{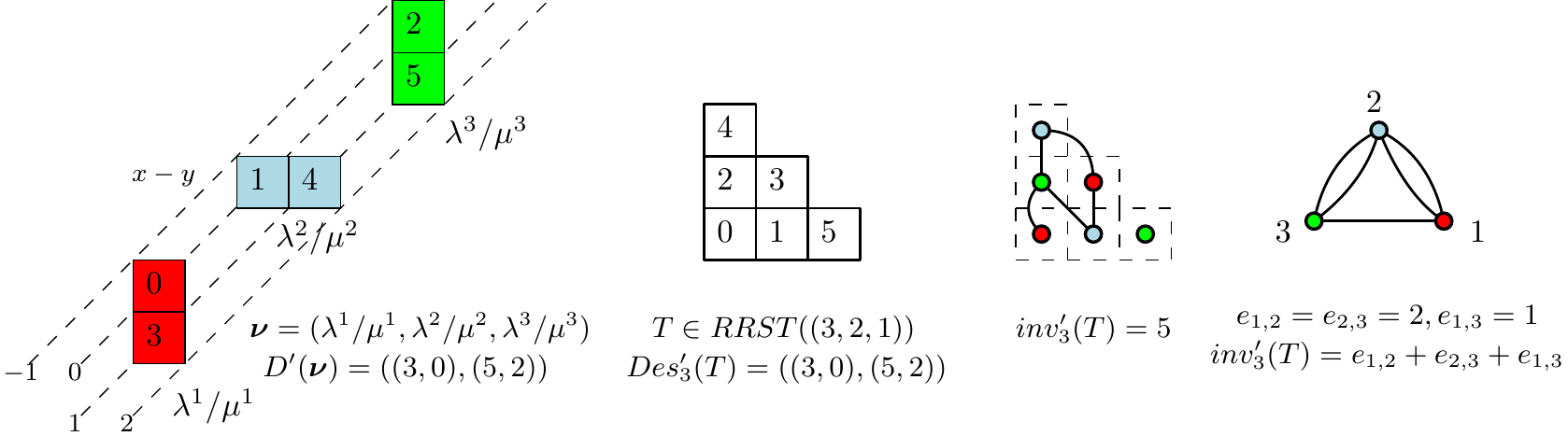}
  \caption{Restricted square strict tableau corresponding to
    Schur-expansion of an LLT polynomial of three skew shapes.}
\label{fig:Blasiak}
\end{figure}
\end{example}

Using the notation from \cref{fig:Blasiak}, let $e_{i,j}(T)$ denote the
number of pairs $(\square,\square')$ contributing to $\inv_3'(T)$ with
$T(\square)\equiv i$ and $T(\square')\equiv j$ modulo $3$, so that
\[\inv_3'(T) = e_{1,2}(T) + e_{1,3}(T) + e_{2,3}(T).\]
We believe that the following is true:
\begin{conjecture} \label{conj:auxiliaryforcumulantof3}
	For any triple of skew diagrams $\nnu = (\la^1/\mu^1,\la^2/\mu^2,\la^3/\mu^3)$ and every triple $\{i,j,k\} = \{1,2,3\}$ with $i < j$, we have
	\begin{equation} \label{eq:prop3shapes}
	\LLT(\la^i/\mu^i,\la^j/\mu^j)(q)\cdot \LLT(\la^k/\mu^k)(q) = \sum_\la \left(\sum_{\substack{T\in\RSST(\la) \\ \Des_3'(T) = D'(\beta) \\ c(\nnu) \text{ - entries of } T}} q^{e_{i,j}(T)}\right) s_\la.
	\end{equation}
      \end{conjecture}

\begin{corollary}
	Assume that \cref{conj:auxiliaryforcumulantof3} holds
        true. Then \cref{conj:LLTSchurPos} holds true for all
        $r$-colored triples of skew shapes.
\end{corollary}
\begin{proof}
	The proof follows the same argument as the one used in \cref{thm:kumuMonoPos} to show that
	\[ [s_\la]\ka_{\LLT}(\nnu,f) = \sum_{\substack{T\in\RSST(\la) \\
			\Des_3'(T) = D'(\beta) \\ c(\nnu) \text{ - entries of }
			T}}\mathcal{I}_{(G^T)_f}(q),\]
	where $G^T$ is an $f$-colored graph whose vertices are entries of $T$
	and we connect pairs contributing to $\inv_3'(T)$.
\end{proof}

Note that the above argument works for any $r$-colored tuple of shapes
$(\nnu,f)$ and thus, \cref{conj:LLTSchurPos} suggests the
following interesting structure of the coefficients of LLT-polynomials
in the Schur expansion.

\begin{problem}
  \label{prob:Structure}
  Let $\nnu = (\la^1/\mu^1,\dots,\la^r/\mu^r)$ be an $r$-tuple of
  skew Young diagrams. Is it true that for any partition $\la$ there
  exists a class of graphs $\mathcal{G}^{\nnu}_\la$ with the set
  of vertices $[1..r]$ such that for any set-partition $\pi \in
  \PPP([1..r])$ one has
  \[ [s_\la]\prod_{B \in \pi}\LLT^{\cospin}((\nnu)^B) = \sum_{G \in
      \mathcal{G}^{\nnu}_\la}q^{\sum_{B\in \pi}e_B},\]
  where $(\nnu)^B := (\nnu,\id_{[1..r]})^B$ and
  $\id_{[1..r]}\colon [1..r] \to [1..r]$ is the identity function?
\end{problem}

Note that the affirmative answer for this problem implies
\cref{conj:LLTSchurPos} providing its combinatorial interpretation:
\[ [s_\la]\ka_{\LLT^{\cospin}}(\nnu,f) = \sum_{G \in
    \mathcal{G}^{\nnu}_\la}\mathcal{I}_{(G)_f}(q).\]

In the next section, we show that \cref{prob:Structure} has an affirmative answer in
some special cases and thus, \cref{conj:LLTSchurPos} holds true for
them.

\subsection{Unicellular LLT and melting lollipops}
\label{subsec:UniLLT}

  \begin{figure}[]
  \includegraphics[width=\linewidth]{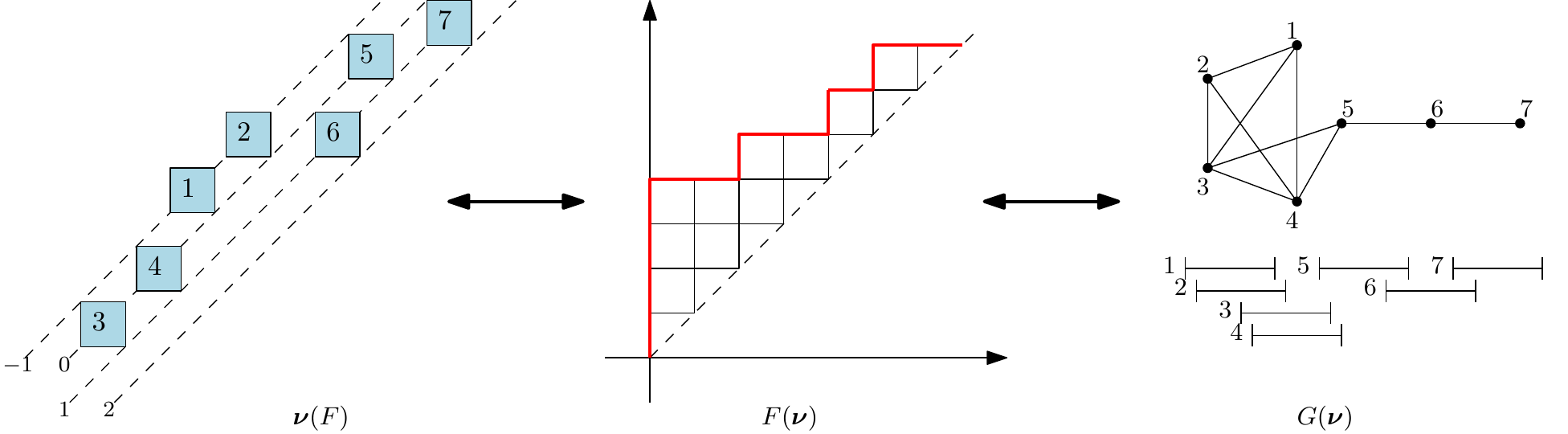}
  \caption{The correspondence between unicellular LLT polynomials,
    Dyck paths and unit interval graphs. The graph $G(\nnu)$ on
    the right is the melting lollipop graph $L^{(2)}_{(5,2)}$ and we
    display the arrangment of unit intervals which realizes it as the
    unit interval graph.}
\label{fig:Melting}
\end{figure}

  A \emph{Schr\"oder path} of length $n$ is a path from $(0,0)$ to $(n,n)$
  composed of steps $\uparrow\ = (0,1)$, $\rightarrow\ = (1,0)$, and
  $\nearrow\ = (1,1)$ (referred to as \textit{north}, \textit{east},
  and \textit{diagonal steps}, respectively), which stays above the
  main diagonal (i.e., it can touch it, but the diagonal steps lie
  strictly above it). Denote by $(i,j)$ the coordinates of the
  $1\times 1$ box with upper right vertex in $(i,j)$. It is well
  known~\cite{Haglund2008} that the vertical-shape LLT polynomials of
  homogenous degree $n$ are
  in bijection with Schr\"oder paths of length $n$: start from an
  $\ell$-tuple of vertical shapes $\nnu$ of size $n$, label its
  boxes by their shifted contents and standardize them, i.e. replace
  them (in the unique way) by labels in $[1..n]$ such that the order of new labels is the same as the order of shifted
  contents. Now construct a Schr\"oder path $F(\nnu)$ such that
  the box $(i,j)$ lies below the path if and only if the entry
  $i$ attacks the entry $j$ in $\nnu$ and the box $(i,j)$
  lies on the diagonal step if the entry $j$ lies directly below
  the entry $i$. This procedure is clearly invertible and we denote by
  $\nnu(F)$ the tuple of vertical strips associated with the
  Schr\"oder path $F$ (see the left side of~\cref{fig:Melting} and consult~\cite{Haglund2008}
  for more details).

  \vspace{5pt}
  
  A special case of a Schr\"oder path is a
  \emph{Dyck path}, that is a path with no diagonal steps. The
  corresponding $\ell$-tuple of vertical shapes $\nnu$ of size
  $n$ is a sequence of $n$ single boxes (i.e.~$\ell=n$) and its LLT
  polynomial is called \emph{unicellular}. It is remarkable that the
  LLT graphs associated with unicellular LLT polynomials are
  precisely \emph{unit interval graphs}, i.e. they can be realized as
  the intersection graphs of $n$ unit intervals on the line (see the right side of~\cref{fig:Melting}).

  Note that for every unit interval graph $G$ on $n$ vertices, one has
  \[ \LLT(G)(1) = e_1^n = \sum_{\lambda \vdash n}\sum_{T \in
      \SYT(\lambda)}s_\lambda.\]
  Therefore, it is natural to look for a statistic
  $a_G:\SYT \to \N$ such that
  \[ [s_\lambda]\LLT(G)(q) = \sum_{T \in \SYT(\lambda)}q^{a_G(T)}.\]
  Recall that the descent set $\Des(T)$ of a
  standard Young tableau $T \in \SYT(\lambda)$ is given by the values
  $i \in [1..n-1]$ for which the entry
  $i+1$ lies in $T$ in a row above the entry $i$\footnote{it is easy
    to check that this definition coincides with the previous
    definition of $\Des(T)$ given in \cref{subsub:Fundamental} in the special case of $\nnu = (\lambda)$ and $T \in \SYT(\nnu)$} and define
  \[ \overleftarrow{\Des(T)} := \{n+1-i\colon i \in \Des(T)\}.\]
  
  Let $m \geq 1,n$ be nonnegative integers and $0 \leq k \leq m-1$. A \emph{melting
  	lollipop} $L_{(m,n)}^{(k)}$ is a graph with the vertex set
  $[1..m+n]$, 
  built by joining the complete graph on vertices $[1..m]$ with the
  path on vertices $[m..m+n]$ (with edges of the form $(i,i+1)$) and
  erasing edges $(1,m),(2,m),\dots,(k,m)$. The unit interval graph
  depicted in~\cref{fig:Melting} is the melting lollipop
  $L^{(2)}_{(5,2)}$.
  
  Recently Huh, Nam
  and Yoo proved the following theorem~\cite{HuhNamYoo2020}:

  \begin{theorem}{\cite{HuhNamYoo2020}}
    \label{theo:HuhNamYoo}
    Let $\mathcal{F}_n$ be the family of unit interval graphs with $n$
    vertices such that
    \[ \LLT(G)(q) = \sum_{\lambda \vdash n}\sum_{T \in
      \SYT(\lambda)}q^{\sum_{i \in \overleftarrow{\Des(T)}}\deg^G_{\operatorname{in}}(i)}s_\lambda\]
    for each $G \in \mathcal{F}_n$ (here $\deg^G_{\operatorname{in}}(i)$
    denotes the number of edges in $G$ incoming to the vertex $i$). Then $\mathcal{F}_n$ contains
    melting lollipops and their disjoint unions.
    \end{theorem}

  Melting lollipops contain two extremal
  cases for which \cref{theo:HuhNamYoo} is a classical result:
  the complete graph $K_n = L_{(n,0)}^{(0)}$ and the path graph $P_n =
  L_{(1,n-1)}^{(0)} = L_{(2,n-2)}^{(0)}$.

  \begin{theorem}
    \label{theo:Lollipop}
    Let $G$ be a melting lollipop graph with $r$ vertices. Then for
    every set-partition $\pi \in \PPP([1..r])$, one has
      \[ [s_\la]\prod_{B \in \pi}\LLT^{\cospin}(G|_B) = \sum_{T \in
      \SYT(\la)}q^{\sum_{B\in \pi}e_B(G_{\operatorname{in}}^{\overleftarrow{\Des(T)}})},\]
    where $G_{\operatorname{in}}^A$ is a graph obtained from $G$ by
    removing all the edges which are not incoming to vertices in $A
    \subset V$.
    In particular, \cref{prob:Structure} and \cref{conj:LLTSchurPos} have an affirmative answer in
    this case and
    \[ [s_\la]\ka_{\LLT^{\cospin}}(G,f) = \sum_{T \in
      \SYT(\la)}\mathcal{I}_{(G_{\operatorname{in}}^{\overleftarrow{\Des(T)}})_f}(q).\]
\end{theorem}

\begin{proof}
  It is enough to notice that
  \begin{itemize}
    \item for every set-partition $\pi \in \PPP([1..r])$ the
      graph $G_\pi := \bigoplus_{B \in \pi}G|_{V_B}$ is a disjoint union of melting
      lollipops so that
          \[ \prod_{B \in \pi}\LLT(G|_{V_B})(q) = \LLT(G_\pi)(q) = \sum_{\lambda \vdash n}\sum_{T \in
      \SYT(\lambda)}q^{\sum_{i \in \overleftarrow{\Des(T)}}\deg^{G_\pi}_{\operatorname{in}}(i)}s_\lambda;\]
  \item the identity
    \[\sum_{i \in A}\deg^{G}_{\operatorname{in}}(i) =
      |E(G_{\operatorname{in}}^A)|\]
     follows directly from the construction of $G_{\operatorname{in}}^A$.
    \end{itemize}
  \end{proof}
  
  \begin{remark}
Note that the class $\mathcal{F}_n$ is strictly smaller than the class
of unit interval graphs on $n$ vertices which can be seen already for
$n=4$: the unit interval graph $G = (V=[1..4],E)$ with $E =
\{(1,2),(2,3),(2,4),(3,4)\}$ does not belong to
$\mathcal{F}_n$. On the other hand, we were not able to find any graph
which belongs to $\mathcal{F}_n$ and is not a disjoint union of
melting lollipops, and it is tempting to conjecture that these two
classes of graphs coincide.
  \end{remark}

We finish by discussing a different approach to attacking
\cref{conj:LLTSchurPos}. One can try to find an
explicit formula for $\ka_{\LLT^{\cospin}}(\nnu,f)$ as a linear
combination of LLT polynomials with coefficients in
$\mathbb{Z}_{\ge 0}[q,q^{-1}]$. Note that Schur polynomials are a special case
  of LLT polynomials so \cref{conj:LLTSchurPos} claims that such an
  expression exists. Nevertheless, we want to stress out that LLT
  polynomials are not linearly independent so one can
  hope that some expressions are more natural and easier than others. One particular example
  where we observed such a natural combinatorial expression is the
  unicellular case corresponding to the complete graph,
  i.e., when $\nnu$ is an $r$-colored tuple of $r$ single boxes: $\la^i=(1), \mu^i = \emptyset$ for all
  $1 \leq i \leq r$. This case might seem to be trivial at first
  sight, but one can quickly convince oneself that this is
  a false impression. It turned
  out that the corresponding cumulant involves beautiful combinatorial objects such as parking functions and
  it has a form similar to the formula in the Shuffle Theorem, conjectured in~\cite{HaglundHaimanLoehrRemmelUlyanov2005} and
  proved by Carlson and Mellit~\cite{CarlssonMellit2018}. Before we
  show the formula we quickly explain what parking functions are.

    \begin{figure}[]
  \includegraphics[width=\linewidth]{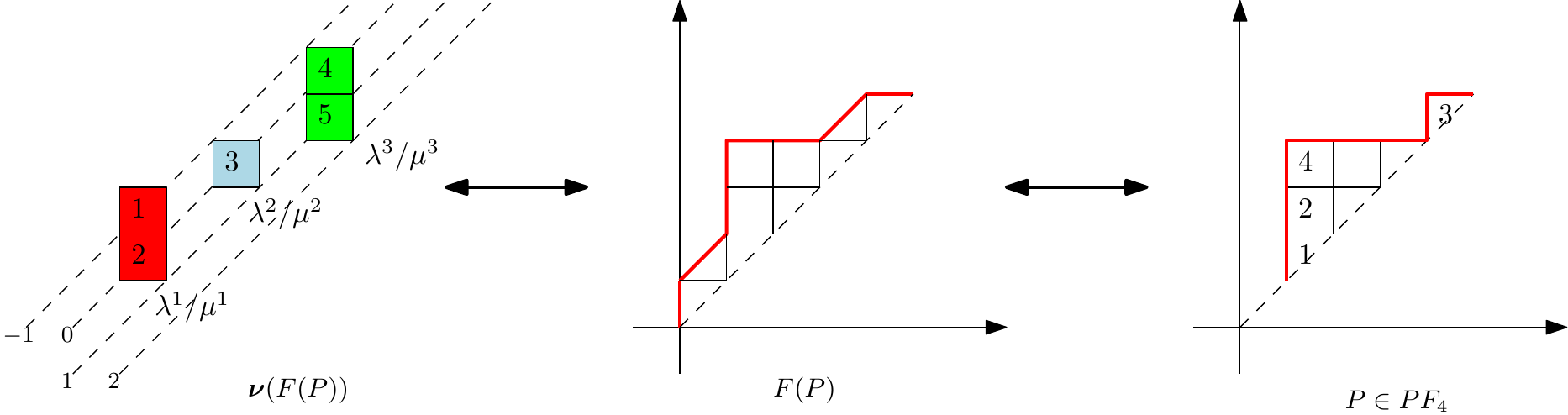}
  \caption{The correspondence between tuples of vertical shapes,
    Schr\"oder paths and parking functions.}
\label{fig:Schroder}
\end{figure}

  A \emph{parking function} $P \in \PF_{n-1}$ of size $n-1$ is a
  function $P\colon [1..n-1]\to [1..n-1]$ such that for each $i \in
  [1..n-1]$, one has $|P^{-1}(i)| \geq i$. One can represent a parking
  function by drawing a Dyck path from $(1,1)$ to $(n,n)$ and
  labeling the boxes to the right of north steps by distinct integers
  $[1..n-1]$ in such a way that the labels of boxes stacked in the same column
  are upward increasing. Starting from a parking function $P
  \in \PF_{n-1}$, convert the corresponding Dyck path of length $n-1$
  into a Schr\"oder path of length $n$ by adding steps
  $\uparrow,\rightarrow$ starting from $(0,0)$ and then replacing all the pairs of
  consecutive steps $(\rightarrow,\uparrow)$ by $\nearrow$, see the
  right side of~\cref{fig:Schroder}. The following formula
  was recently proved by the second author:                                              

  \begin{theorem}{\cite{Kowalski2020}} \label{thm:singcelldecomp}
Let $(\nnu,f) = (((1),\dots,(1)),f)$ be an $r$-colored tuple of
$r$ single boxes. Then
	\begin{equation} \label{thm:formulawithtrees}
	\ka_{\LLT}(\nnu,f)= \ka_{\LLT^{\cospin}}(\nnu,f) = \sum_{P\in \PF_{r-1}} \LLT(\nnu(F(P))),
	\end{equation}
	where we sum over all parking functions of size $r-1$.
      \end{theorem}

This formula gives a positive expression in terms
of vertical-shaped LLT polynomials, which are Schur-positive (by~\cite{GrojnowskiHaiman2007}) and $e$-positive after applying the shift
(by~\cite{DAdderio2020,AlexanderssonSulzgruber2022}). In particular,
\cref{thm:singcelldecomp} gives yet another proof of \cref{conj:LLTSchurPos} and also
\cref{conj:MacSchurPos} in
this special case. Although \cref{thm:singcelldecomp} might suggest
that there is a combinatorial formula expressing an LLT
cumulant as a positive combination of LLT polynomials, we were not able
to find a pattern allowing us to construct such a formula in general and
we leave this problem for further investigations in the future.

    \section*{Acknowledgments}
    MD would like to thank Erik Carlsson and Fernando Rodriguez
    Villegas for interesting discussion on possible
    connections between Macdonald cumulants and the work of Hausel,
    Letellier and Rodriguez
    Villegas~\cite{HauselLetellierRodriguezVillegas2011}.

\bibliographystyle{amsalpha}

\bibliography{biblio2015}

\end{document}